\documentclass[12pt]{article}

\usepackage[margin=1in]{geometry} 
\usepackage[utf8]{inputenc}
\usepackage{amsthm}			%basic math (proof environment)   
\usepackage{amsmath}
\usepackage{amssymb}      
\usepackage{enumitem}      
\usepackage{url}           
\usepackage{algpseudocode}
\usepackage{algorithm}
\usepackage{mathtools}
\usepackage{subcaption}
\mathtoolsset{showonlyrefs}
\usepackage{graphicx}
\usepackage{epstopdf} 
\usepackage{mathptmx} 
\usepackage[usenames,dvipsnames]{xcolor}

\newtheorem{problem}{Problem}
\newtheorem{theorem}{Theorem}
\newtheorem{proposition}[theorem]{Proposition}
\newtheorem{corollary}[theorem]{Corollary}
\newtheorem{lemma}[theorem]{Lemma}
\newtheorem{definition}[theorem]{Definition}

\newcommand{\mcal}{\mathcal{M}}
\newcommand{\ucal}{\mathcal{U}}
\newcommand{\tcal}{\mathcal{T}}
\newcommand{\ical}{\mathcal{I}}
\newcommand{\Rbb}{\mathbb{R}}
\newcommand{\Nbb}{\mathbb{N}}
\newcommand{\pa}[1]{\left(#1\right)}
\newcommand{\pac}[1]{\left[#1\right]}
\newcommand{\paa}[1]{\left\{#1\right\}}
\newcommand{\papc}[1]{\left(#1\right]}
\newcommand{\pacp}[1]{\left[#1\right)}
\newcommand{\norm}[1]{\left\lVert#1\right\rVert}
\newcommand{\restr}[1]{{\big\vert}_{#1}}
\newcommand{\scalp}[2]{\left\langle #1,#2 \right\rangle}
\newcommand{\map}[5]{\begin{aligned}#1 \,:\, #2 &\to #3 \\ #4 &\mapsto #5\end{aligned}}
\newcommand{\oneOver}[1]{\frac{1}{#1}}
\newcommand{\dd}{\mathrm{d}}
\newcommand{\D}{\mathrm{D}}
\newcommand{\Ddt}[2]{\frac{\D #1}{\dd #2}}
\newcommand{\ddt}[2]{\frac{\dd #1}{\dd #2}}

\newcommand{\Log}{\mathrm{Log}}
\newcommand{\Exp}{\mathrm{Exp}}
\newcommand{\piOnM}{\pi_\mcal}
\newcommand{\piOnT}[1]{\pi_{T_{#1}\mcal}}

\newcommand{\rank}[1]{\mathrm{rank}\pa{#1}}
\newcommand{\stiefel}[2]{\mathrm{St}\hspace*{-0.07cm}\pa{#1,#2}}

\newcommand{\matr}[2]{\Rbb^{#1 \times #2}}
\newcommand{\supOn}[1]{\underset{#1}{\sup}}
\newcommand{\infOn}[1]{\underset{#1}{\inf}}

\newcommand{\floor}[1]{\left \lfloor{#1}\right \rfloor }
\newcommand{\convergesTo}[2]{\overset{#1\to #2}{\longrightarrow}}
\newcommand{\operatornorm}[1]{\norm{#1}_{\mathrm{op}}}
\newcommand{\parallelTransp}[3]{\mathrm{P}_{#1 \to #2}^{#3}}

\newboolean{ShowRevisions}  
\setboolean{ShowRevisions}{false}   
\ifthenelse{\boolean{ShowRevisions}}
{	
	\newcommand{\revision}[1]{\textcolor{blue}{#1}}
}
{
	\newcommand{\revision}[1]{#1}
}

\begin{document}
	
	\title{Hermite interpolation with retractions on manifolds}
	\author{Axel Séguin\thanks{École Polytechnique Fédérale de Lausanne (EPFL) Institute of Mathematics, CH-1015 Lausanne, Switzerland (axelseguin95@gmail.com, daniel.kressner@epfl.ch)} \and Daniel Kressner\footnotemark[1] }
	\date{}
	
	\maketitle
	
	\begin{abstract}	
	Interpolation of data on non-Euclidean spaces is an active research area fostered by its numerous applications. This work considers the Hermite interpolation problem: finding a sufficiently smooth manifold curve that interpolates a collection of data points on a Riemannian manifold while matching a prescribed derivative at each point. A novel procedure relying on the general concept of retractions is proposed to solve this problem on a large class of manifolds, including those for which computing the Riemannian exponential or logarithmic maps is not straightforward, such as the manifold of fixed-rank matrices. The well-posedness of the method is analyzed by introducing and showing the existence of retraction-convex sets, a generalization of geodesically convex sets. A classical result on the asymptotic interpolation error of Hermite interpolation is extended to the manifold setting. Finally numerical experiments on the manifold of fixed-rank matrices and the Stiefel manifold of matrices with orthonormal columns illustrate these results and the effectiveness of the method.
	\end{abstract}
	
	\section{Introduction}
	
	Data processing on non-Euclidean spaces has become a well-established tool in many fields of science and engineering. In particular, there has been a rising interest to interpolate data on a manifold with a curve contained in the manifold. This is motivated by numerous applications in robotics~\cite{parkRavani}~\cite{blochCamarinha}, computer vision~\cite{krakowskiApplied}, medical imaging~\cite{gousenbourgerCurve}~\cite{kim}, statistics~\cite{massartAbsil} and model-order reduction~\cite{amsallem}, just to mention a few. For example, motion-planning of a robotic manipulator can be carried out by interpolating points on the Lie group of rigid motions SE(3)~\cite{parkRavani}. In statistical modeling, estimating covariance matrices between discrete samples of a random field can be viewed as interpolation on the manifold of symmetric positive definite matrices~\cite{massartAbsil}. Reduced-order modeling in engineering typically involves projecting high-dimensional dynamics onto low-dimensional subspaces and interpolating such subspaces on the Grassman manifold is an important task for parameter-dependent systems~\cite{amsallem}. 
	
	Over the last two decades, several ways of performing manifold interpolation computationally have been proposed. These methods are tailored to meet different requirements of the application, concerning the regularity of the interpolating curve and the nature of the interpolation constraints.
	In this present work, we focus on continuously differentiable interpolation curves that match prescribed data \emph{and}
	velocities at each point; this is commonly known as \emph{Hermite interpolation}.

	\paragraph{Related work.} In Euclidean space, the classical solution of the Hermite interpolation problem utilizes piecewise cubic polynomials~\cite{farinBook}. This solution can be characterized in (at least) three different ways: (i) it minimizes the integral of the squared acceleration over the set of admissible interpolating curves, (ii) it is the unique piecewise cubic polynomial interpolating the points and the derivatives, (iii) it can be constructed with a geometric algorithm introduced by de Casteljau~\cite{deCastPaper} involving iterated linear interpolation between suitably chosen control points. 
	As we explain in the following, each of these characterizations can be generalized to manifolds. However, unlike for the Euclidean case, each characterization and extension produces a different curve.
	
	(i) Extending the variational characterization is straightforward: the second derivative (acceleration) is replaced by the covariant derivative of the velocity vector and the search space is constrained to curves on the manifold~\cite{camarinhaLeite}~\cite{zhangNoakes} that satisfy the interpolation constraints. The solution to this variational problem, also known as Riemannian cubic, can be computed by numerical methods for boundary value differential problems, such as shooting methods~\cite{blochCamarinha}.
	
	(ii) Exploiting the polynomial characterization of the solution requires one to recast manifold interpolation into an Euclidean setting. Several strategies have been explored for this purpose, e.g., a  local linearization can be obtained from a (local) bijection between the manifold and a tangent space. For Hermite interpolation, Zimmermann~\cite{zimmermann} proposes to use the Riemannian logarithmic map and its differential to map points and derivatives to the tangent space at one of the data points. Standard techniques, including polynomial interpolation, can then be used to construct the interpolation curve on the tangent space and map it back to the manifold with the Riemannian exponential map.  We refer to~\cite{zimmermann2022} for a recent extension to the multivariate setting.
	
	(iii) The geometric nature of the de Casteljau algorithm lends itself to an intrinsic definition on  manifolds. In fact, transplanting into a manifold setting simply requires to replace
	the linear segments that define the algorithm with geodesic segments, noting that geodesics are the manifold generalization of straight lines. This idea was pioneered by Park and Ravani in~\cite{parkRavani}; it has been adapted to solve the Hermite interpolation problem on spheres and compact Lie groups by Crouch and Leite~\cite{crouchLeite}  and on general Riemannian manifolds by Popiel and Noakes \cite{popielNoakes}. Rodriguez et al.~\cite{rodriguez} proposed a similar approach, where the classical de Casteljau algorithm is interpreted and generalized as the weighted average of two curve branches satisfying the interpolating conditions. This blending technique is adapted to the manifold setting using geodesics branches and geodesic averaging, which -- at least in principle -- allows one to perform Hermite interpolation on a large class of manifolds even if the work~\cite{rodriguez} itself focuses on compact Lie groups and the sphere.
	
	The non-exhaustive list of algorithms above aims at illustrating that most approaches proposed so far focus on relatively simple manifolds: compact Lie groups, symmetric spaces like the sphere or complete Riemannian manifolds. Most of these techniques require to have closed-form expressions or at least a numerically tractable method for computing endpoint geodesics or the Riemannian exponential and logarithmic maps. For instance, for the case of the Stiefel manifold, the exponential map under the canonical metric has a closed-form expression \cite[\S 2.4.2]{edelmanExpForStiefel} and an algorithm to approximate the corresponding logarithmic map has been proposed in \cite{zimmermannLogMap}. In contrast, for the manifold of fixed-rank matrices we are not aware of a computationally efficient way to realize the logarithmic map; a closed-form expression for the exponential map (under a suitable quotient geometry) is given in~\cite[\S 6]{absilLowRankExp}.
	
	A first step towards relaxing computational requirements has been put forth by Polthier and Nava-Yazdani~\cite{navaYazdaniPolthier} by generalizing the de Casteljau algorithm to work with generic endpoint curves instead of geodesics as building blocks. This enlarges the applicability of the algorithm to, e.g., polygonal surfaces. On the other hand, velocity constraints cannot be taken into account and the concatenation of two such generalized de Casteljau curves may result in a non-differentiable junction. These drawbacks of~\cite{navaYazdaniPolthier} were addressed by Krakowski et al.~\cite{krakowski} for the case of the Stiefel manifold with a generalized de Casteljau algorithm that uses a novel class of endpoint curves, termed quasi-geodesics. This technique allows one to interpolate points on the Stiefel manifold with a globally $C^1$ curve but the velocity can only be prescribed at the starting point.  
	
	\paragraph{Contribution.} In this work, we propose and analyze a new Hermite interpolation technique on Riemannian manifolds. Extending upon~\cite{krakowski,navaYazdaniPolthier} in being able to handle velocity constraints, our approach utilizes retractions to construct a novel class of endpoint curves. Retractions can be thought as first-order approximations of the exponential map that are widely used in Riemannian optimization~\cite{Absil2008,boumalBook}. For a large class of manifolds and retractions, the inverse retraction is available in closed-form. Whenever this is the case, our method makes it possible to solve the Hermite interpolation problem in a numerically efficient way on a larger class of manifolds than those available until now.  
	
	\paragraph{Outline.} In Section~\ref{s:generalizedDeCast}, we develop a retraction-based Hermite (RH) interpolation scheme on manifolds. The well-posedness of the method is guaranteed on retraction-convex sets, a generalization of geodesically convex sets that we define and develop in Section~\ref{s:retractionConvexSets}. A convergence analysis of the RH scheme is carried out in Section~\ref{s:analysis}, extending well-known results on the convergence of Hermite interpolation on Euclidean space for sufficiently smooth data. Finally, in Section~\ref{s:numericalExperiments}, we demonstrate several applications of our novel interpolation scheme for both the manifold of fixed-rank matrices and the Stiefel manifold. 
	
	\paragraph{Mathematical setting.}
	We recall some basic definitions of Riemannian geometry needed throughout this paper; see \cite{leeRiemManifs2} for details. Let $\mcal$ denote a $D$-dimensional connected manifold endowed with a Riemannian metric $g$ and the corresponding Levi-Civita connection $\nabla$. The tangent space at $x\in\mcal$ is denoted by $T_x\mcal$ and ${T\mcal := \paa{(x,v): x\in\mcal, v\in T_x\mcal}}$ denotes the tangent bundle. When $\mcal$ is an embedded Riemannian submanifold, we let $N_x\mcal$ denote the normal space at $x\in\mcal$ and $\Pi_x$ the orthogonal projection onto the tangent space at $x$.
	The Riemannian metric defines an inner product $\scalp{\cdot}{\cdot}_x$ on each tangent space $T_x\mcal$, with the induced norm $\norm{\cdot}_x$. We let ${d:\mcal\times\mcal\to\pacp{0,+\infty}}$ denote the Riemannian distance function defined for any $x,y\in\mcal$ by 
	\begin{equation}\label{eq:distanceFunction}
		d(x,y) := \underset{\delta\in\Gamma_{x,y}}{\inf} L(\delta),
	\end{equation}
	where $\Gamma_{x,y}$ contains every piecewise differentiable curve $\delta: \pac{0,1} \to \mcal$ joining $x$ and $y$, and $L(\delta) := \int_{0}^{1}\|\dot\delta(\tau)\|_{\delta(\tau)}\dd\tau$ denotes its length. This also allows us to define the open ball centered in $x$ of radius $r$ as
	\begin{equation} \label{eq:manifoldball}
		B(x,r) := \paa{y\in\mcal: d(x,y)<r}.
	\end{equation}
	
	For $x\in\mcal$, the Riemannian exponential $\Exp_x:T_x\mcal\to\mcal$ maps a tangent vector $v\in\ T_x\mcal$ to $\gamma_{x,v}(1)$, where $\gamma_{x,v}: [0,1]\to \mcal$ denotes the unique geodesic such that $\gamma_{x,v}(0) = x$ and $\dot\gamma_{x,v}(0) = v$. In general, such a geodesic is well-defined (and smooth) only for $v$ in a neighborhood of $0 \in T_x\mcal$ and locally invertible. The Riemannian logarithm $\Log_x$ is defined as the local inverse of the Riemannian exponential.  
	
	The mathematical formulation of the manifold Hermite interpolation problem is the following.
	\begin{problem}\label{problem:HermiteInterpolation}
		Given $N+1$ tangent bundle data points $\paa{(p_i,v_i)}_{i = 0}^N\in T\mcal$ and scalar parameters $t_0<t_1<\dots<t_N$, find a continuously differentiable curve $H:\pac{t_0,t_N}\to\mcal$ such that 
		\begin{equation}
			\begin{cases}
				H(t_i) = p_i,\\
				\dot H(t_i) = v_i,
			\end{cases}
			\quad \forall i = 0,\dots,N.
		\end{equation}
	\end{problem}

	\section{Generalized de Casteljau algorithm with retractions}\label{s:generalizedDeCast}
	
	The classical de Casteljau algorithm~\cite{deCastPaper} is a geometric procedure to construct polynomial curves in $\Rbb^D$. To describe the algorithm, let $\sigma_1(t;x,y) := (1-t)x + ty$ denote 
	the linear interpolation between two points $x,y\in \Rbb^D$.
	Given $N+1$ so called \emph{control points} ${b_0,\dots,b_N \in \Rbb^D}$,  the relation 
	\[
	\sigma_k(t;b_i,\dots,b_{i+k}) := \sigma_1(t;\sigma_{k-1}(t;b_i,\dots,b_{i+k-1}), \sigma_{k-1}(t;b_{i+1},\dots,b_{i+k})), \quad i = 0,\dots,N-k,
	\]
	is applied recursively for $k = N,N-1,\ldots, 2$ to define a polynomial curve $\sigma_{N}$ of degree $N$ such that
	\begin{equation}
		\begin{aligned}
			\sigma_{N}(0;b_0,\dots,b_N) &= b_0,\\
			\sigma_{N}(1;b_0,\dots,b_N) &= b_N,\\
			\dot\sigma_{N}(0;b_0,\dots,b_N) &= N\dot\sigma_1(0;b_0,b_1) =  N(b_1-b_0),\\
			\dot\sigma_{N}(1;b_0,\dots,b_N)
			&= N\dot\sigma_1(1;b_{N-1},b_{N}) = N(b_N-b_{N-1}),
		\end{aligned}
	\end{equation}
	where $\dot\sigma_N$ denotes the derivative of $\sigma_N$ with respect to $t$.
	Only the first and last control points are interpolated while the other points influence the shape of the curve. In~\cite{parkRavani} an extension to control points on a manifold $\mcal$ was proposed that replaces $\sigma_1$ by the endpoint geodesic joining $x$ and $y$. The same recursive relation then yields a manifold curve verifying analogous properties: (a) it is smooth, (b) it interpolates the first and last points, and (c) the derivatives at the first and last points only depend on the first and last two control points, respectively. Property (c) allows one to conveniently control the endpoint derivatives via the choice of $b_1,b_{N-1}$, a property that makes the de Casteljau algorithm useful for Hermite interpolation; see, e.g.,~\cite{popielNoakes}. 
	
	Along the line of work by Krakowski et al. \cite{krakowski}, we consider a generalization of the de Casteljau algorithm that allows for arbitrary smooth manifold curves at each step of the recursion (instead of constructing everything on top of geodesics). If each curve chosen to define the algorithm joins the prescribed endpoints then properties (a) and (b) are trivially satisfied. In \cite[Proposition 13]{krakowski}, sufficient conditions for the chosen curves to produce property (c) are given for $N=2$. The following proposition extends these results to $N=3$. 
	\begin{figure}
		\centering
		\includegraphics[width=0.6\linewidth]{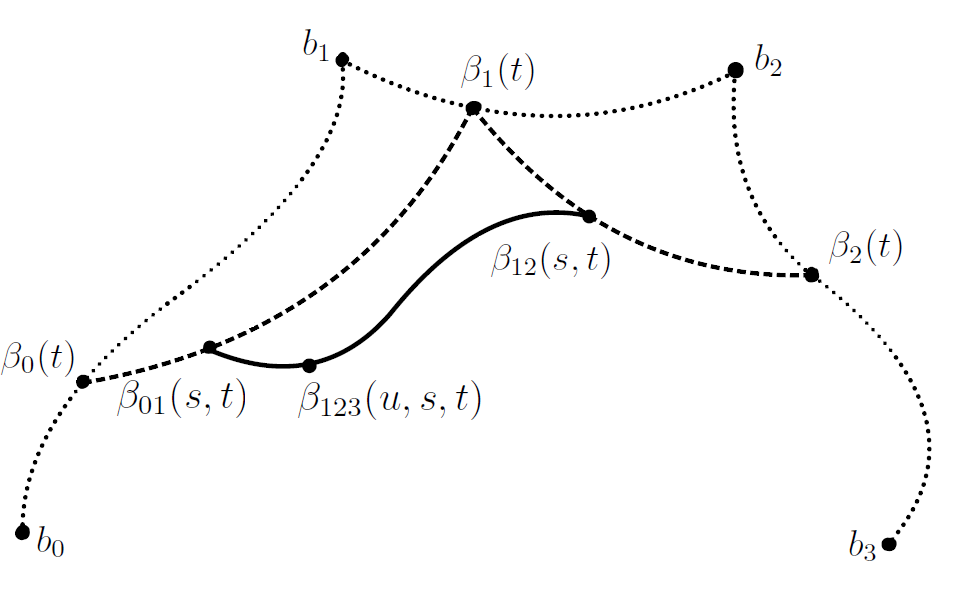}
		\caption{Generalized de Casteljau with 4 control points.}
		\label{fig:myDeCast}
	\end{figure}
	\begin{proposition}\label{prop:generalizedDeCast}
		For $b_0, b_1, b_2, b_3\in \mcal$ consider, as in Figure~\ref{fig:myDeCast}:
		\begin{itemize}[label=-]
			\item smooth $\beta_i:\pac{0,1}\to\mcal$ joining $b_i$ and $b_{i+1}$ for each $i=0,1,2$, 
			\item smooth $\beta_{01}:\pac{0,1}^2\to\mcal$ such that $\beta_{01}(\cdot,t)$ joins $\beta_0(t)$ and $\beta_1(t)$ for every $t\in [0,1]$,
			\item smooth $\beta_{12}:\pac{0,1}^2\to\mcal$ such that $\beta_{12}(\cdot,t)$ joins $\beta_1(t)$ and $\beta_2(t)$ for every $t\in [0,1]$,
			
			\item smooth $\beta_{012}:\pac{0,1}^3\to\mcal$ such that $\beta_{123}(\cdot,s,t)$ joins $\beta_{01}(s,t)$ and $\beta_{12}(s,t)$ for every $s,t\in [0,1]$.
		\end{itemize}
		If, additionally,
		\begin{enumerate}[label=(\roman*)]
			\item \label{item:beta12} $\beta_{01}(s,0) = \beta_0(s)$ and $\beta_{12}(s,1) = \beta_2(s)$, 
			%\item \label{item:beta23} 
			\item \label{item:beta123a} $\beta_{012}(s,0,0) = \beta_0(s)$ and $\beta_{012}(s,1,1) = \beta_2(s)$,
			%\item \label{item:beta123b} 
		\end{enumerate}
		then the generalized de Casteljau manifold curve 
		\begin{equation}
			\beta(t) = \beta_{012}(t,t,t)
		\end{equation} 
		satisfies  
		\begin{equation} \label{eq:interpolationOfDeCast}
			\beta(0) = b_0,\quad \beta(1) = b_3,
		\end{equation}
		and 
		\begin{equation}\label{eq:derivativeOfDeCast}
			\dot\beta(0) = 3\dot\beta_0(0),\quad \dot\beta(1) = 3\dot\beta_2(1).
		\end{equation}
	\end{proposition}
	
	\begin{proof}
		The interpolation condition~\eqref{eq:interpolationOfDeCast} follows directly from the definitions:
		\[
		\beta(0) = \beta_{012}(0,0,0) = \beta_{01}(0,0) = \beta_0(0) = b_0
		\]
		and, analogously, 
		$
		\beta(1) = b_3.
		$
		To prove~\eqref{eq:derivativeOfDeCast}, we first note that 
		\begin{equation}\label{eq:derivativeOfBeta}
			\dot\beta(t) = \ddt{}{s}\beta_{012}(s,t,t)\restr{s = t} + \ddt{}{s}\beta_{012}(t,s,t)\restr{s = t} + \ddt{}{s}\beta_{012}(t,t,s)\restr{s = t}.
		\end{equation}
		At $t = 0$, inserting the definitions of the curves as well as conditions~\ref{item:beta12} and~\ref{item:beta123a} we thus obtain
		\begin{align}
			\dot\beta(0) & = \ddt{}{s}\beta_{012}(s,0,0)\restr{s = 0} +  \ddt{}{s}\beta_{012}(0,s,0)\restr{s = 0} + \ddt{}{s}\beta_{012}(0,0,s)\restr{s = 0} \\
			& = \dot\beta_0(0) + \ddt{}{s}\beta_{01}(s,0)\restr{s = 0} + \ddt{}{s}\beta_{01}(0,s)\restr{s = 0} = 3\dot\beta_0(0).
		\end{align}
		Analogously, one establishes $\dot\beta(1) = 3\dot\beta_2(1)$, which completes the proof. \qed
	\end{proof}
	
	When one employs the same type of endpoint curve to define each $\beta_0,\dots,\beta_{012}$ in Proposition~\ref{prop:generalizedDeCast}, such as endpoint geodesics or endpoint quasi-geodesics as proposed by \cite{krakowski}, then conditions~\ref{item:beta12} and~\ref{item:beta123a} are trivially satisfied. With this simplification, a result of the form~\eqref{eq:derivativeOfDeCast} for a generalized de Casteljau algorithm of arbitrary order $N$ can be found in \cite[Theorem 8]{navaYazdaniPolthier}. However, to be useful for Hermite interpolation, we need to have explicit relationships between the control points and $\dot\beta_0(0)$ as well as $\dot\beta_2(1)$. While both are available for geodesics, only the first relationship can be made explicit for quasi-geodesics. Consequently, the interpolation problem considered in \cite{krakowski} incorporates a velocity constraint at the first interpolation point only. In the next section we propose a new family of endpoint curves for defining curves $\beta_0,\dots,\beta_{012}$ that satisfy Proposition~\ref{prop:generalizedDeCast} and yield explicit relationships between the control points and $\dot\beta_0(0)$, $\dot\beta_2(1)$.
	
	\label{s:RHinterpolationScheme}
	
	\subsection{A family of endpoint retraction curves}
	We use retractions~\cite[Definition 3.47]{boumalBook} to construct endpoint manifold curves. A retraction $R$ is a smooth map from the tangent bundle $T\mcal$ to the manifold defined in a neighborhood of the origin of each tangent space. Letting $R_x$ denote the restriction of $R$ onto $T_x\mcal$ (that is, $R_x(v) = R(x,v)$), one additionally requires that $R_x(0) = x$ and $\D R_x(0) = \mathrm{I}_{T_x\mcal}$, the identity map of the tangent space.

	The latter condition, also known as local rigidity, guarantees that any retraction is locally invertible; we denote the inverse retraction as $R_x^{-1}$. 
	
	\begin{definition}\label{def:endpointRetr}
		Given $x,y\in\mcal$ and $r\in\pac{0,1}$, the {\emph{$r$-endpoint retraction curve}} joining $x$ and $y$ is defined as
		\begin{equation}
			c_r(t;x,y) = R_{q(r)}\pa{(1-t)R_{q(r)}^{-1}(x) + t R_{q(r)}^{-1}(y)}
		\end{equation}
		where $q(r) = R_x(r R_x^{-1}(y))$.  
	\end{definition}
	The choice of $r\in\pac{0,1}$ determines the anchor point for the retraction curve that defines $c_r$.
	Here and in the following two sections, we assume that any use of the retraction and its inverse is well-defined. For example, this is the case for Definition~\ref{def:endpointRetr} for any $r\in\pac{0,1}$ when
	$x$ and $y$ are sufficiently close. A more precise statement on the well-posedness of $r$-endpoint retraction curves can be found in Proposition~\ref{prop:rEndPointRetrCurveWellPosedness} below.
	\begin{proposition} \label{prop:endPointRetractionCurveProps}
		The $r$-endpoint retraction curve family satisfies the following properties:
		\begin{enumerate}[label=(\roman*)]
			\item \label{item:startsAtX} $c_r(0;x,y) = x$ and $c_r(1;x,y) = y$ for every $r\in\pac{0,1}$;
			
			\item \label{item:controlDerivIn0}$\dot c_0(0;x,y) = R_{x}^{-1}(y)$ and $\dot c_1(1;x,y) = -R_{y}^{-1}(x)$.
			%		\item \label{item:controlDerivIn1}.
		\end{enumerate}
	\end{proposition}
	\begin{proof}
		The definition of $c_r$ directly implies~\ref{item:startsAtX}. To show~\ref{item:controlDerivIn0}, we first note that
		\begin{equation}
			\dot c_r(t;x,y) = \D R_{q(r)}\pa{(1-t)R_{q(r)}^{-1}(x) + t R_{q(r)}^{-1}(y)}\pac{R^{-1}_{q(r)}(y) - R^{-1}_{q(r)}(x)}.
		\end{equation}
		Because $q(0)=x$ and $q(1) = y$, the local rigidity of the retraction yields
		\begin{align}
			&\dot c_0(0;x,y) = \D R_x(0)\pac{R^{-1}_x(y)} =R^{-1}_x(y),\\
			&\dot c_1(1;x,y) = \D R_y(0)\pac{-R_y^{-1}(x)} = -R_y^{-1}(x). 
		\end{align}  \qed
	\end{proof}
	Note that $c_r$ is not invariant under exchanging the endpoints, that is $c_r(t;x,y)$ does not coincide with $c_r(1-t;y,x)$ in general, unless the exponential map itself is chosen as retraction. On the other hand, extreme members of the family are related via $
	c_0(t;x,y) = c_1(1-t;y,x).
	$
	
	We will now explain how $c_r$ is used to define curves that satisfy the conditions of Proposition~\ref{prop:generalizedDeCast} and, in turn, to define a suitable generalization of the de Casteljau algorithm. From Proposition~\ref{prop:endPointRetractionCurveProps}~\ref{item:controlDerivIn0} and \eqref{eq:derivativeOfDeCast}, it follows that $\beta_0(\cdot) = c_0(\cdot;b_0,b_1)$ and $\beta_2(\cdot) = c_1(\cdot;b_2,b_3)$ are canonical choices for joining $b_0$ with $b_1$ and $b_2$ with $b_3$, respectively. The other curves must be suitably chosen from the $r$-endpoint retraction curve family in order to satisfy Proposition~\ref{prop:generalizedDeCast}.
	\begin{proposition}\label{prop:retractionDeCast}
		The following choices of $\beta_0$, $\beta_1$, $\beta_2$, $\beta_{01}$, $\beta_{12}$ and $\beta_{012}$ satisfy the conditions of Proposition~\ref{prop:generalizedDeCast}: 
		\begin{itemize}[label=-]
			\item $\beta_0(t) = c_0(t;b_0,b_1)$, $\beta_1(t) = c_{r_1(t)}(t;b_1,b_2)$, $\beta_2(t) = c_1(t;b_2,b_3)$,
			\item $\beta_{01}(s,t) = c_{r_{01}(s,t)}(s;\beta_0(t),\beta_1(t))$, $\beta_{12}(s,t) = c_{r_{12}(s,t)}(s;\beta_1(t),\beta_2(t))$,
			\item $\beta_{012}(u,s,t) = c_{r_{012}(u,s,t)}(u;\beta_{01}(s,t),\beta_{12}(s,t))$,
		\end{itemize}
		for any smooth functions $r_1:\pac{0,1}\to\pac{0,1}$, $r_{01}, r_{12}:\pac{0,1}^2\to\pac{0,1}$ and $r_{012}:\pac{0,1}^3\to \pac{0,1}$ such that 
		\begin{equation}\label{eq:constraintsOnTheChoiceOfR01ToR012}
			r_{01}(s,0) = 0,\quad 
			r_{12}(s,1) = 1, \quad 
			r_{012}(s,0,0) = 0, \ 
			r_{012}(s,1,1) = 1.
		\end{equation}
		Moreover, the resulting manifold curve $\beta(t) = \beta_{012}(t,t,t)$ satisfies
		\begin{equation}
			\label{eq:derivativeOfRetrDeCast0}\dot\beta(0) = 3\dot c_0(0;b_0, b_1) = 3 R_{b_0}^{-1}(b_1),\quad 
			\dot\beta(1) = 3\dot c_1(1;b_2, b_3) = -3 R_{b_3}^{-1}(b_2).
		\end{equation}
	\end{proposition}
	\begin{proof}
		Proposition~\ref{prop:endPointRetractionCurveProps}~\ref{item:startsAtX} implies that the curves $\beta_0$, $\beta_1$, $\beta_2$, $\beta_{01}(\cdot, t)$, $\beta_{12}(\cdot,t)$ and $\beta_{012}(\cdot, s,t)$ have the correct endpoints for every $s,t\in\pac{0,1}$ and any choice of $r_1$, $r_{01}, r_{12}$ and $r_{012}$. Direct computation shows the remaining requirements~\ref{item:beta12} and \ref{item:beta123a} of Proposition~\ref{prop:generalizedDeCast}:
		\begin{enumerate}[label=(\roman*)]
			\item $\beta_{01}(s,0) = c_{r_{01}(s,0)}(s;\beta_0(0),\beta_1(0)) = c_0(s;b_0,b_1) = \beta_0(s)$, \\
			$\beta_{12}(s,1) = c_{r_{12}(s,1)}(s;\beta_1(1),\beta_2(1)) = c_1(s;b_2,b_3) = \beta_2(s)$,
			\item $\beta_{012}(s,0,0) = c_{r_{012}(s,0,0)}(s;\beta_{01}(0,0),\beta_{12}(0,0)) = c_0(s;b_0,b_1) = \beta_0(s)$, \\
			$\beta_{012}(s,1,1) = c_{r_{012}(s,1,1)}(s;\beta_{01}(1,1),\beta_{12}(1,1)) = c_1(s;b_2,b_3) = \beta_2(s)$.
		\end{enumerate}
		Finally, the relation~\eqref{eq:derivativeOfRetrDeCast0} follows from combining~\eqref{eq:derivativeOfDeCast} with Proposition~\ref{prop:endPointRetractionCurveProps}~\ref{item:controlDerivIn0}. \qed
	\end{proof}
	
	Proposition~\ref{prop:retractionDeCast} offers a great degree of flexibility in choosing $r_1$, $r_{01}, r_{12}$, and $r_{012}$. For practical purposes, a simple choice that leads to a computationally inexpensive evaluation of the curve is preferable. We propose to choose
	\begin{equation}\label{eq:bestChoiceOfr1r01r12r012}
		r_1(s) = 1/2,\quad
		r_{01}(s,t) = 0,\quad 
		r_{12}(s,t) = 1,\quad 
		r_{012}(u,s,t) = t.
	\end{equation}
	\revision{By choosing constant $r_{01}$ and $r_{12}$, the evaluation of each of the intermediate curves $\beta_{01}$ and $\beta_{12}$ requires one retraction and one inverse retraction only for every value of $(s,t)$. For any other choice of $r_{01}$ or $r_{12}$ satisfying Proposition~\ref{prop:retractionDeCast}, these evaluations require two retractions and three inverse retractions.}
	\revision{In total,} as we will see below in Algorithm~\ref{alg:onlinePhase}, the choice~\eqref{eq:bestChoiceOfr1r01r12r012} essentially consists of 7 retractions and 5 inverse retractions per evaluation of the generalized de Casteljau manifold curve, ignoring the cost for preprocessing. 
	
	At this point, the choice of $r_1(s) = 1/2$ appears to be ad hoc, especially because Proposition~\ref{prop:retractionDeCast} imposes no constraint on $r_1$.
	As we will see in Section~\ref{ss:needForBoundedDerivatives}, this choice of $r_1$ is crucial for the scheme to attain favorable convergence properties.
	From now on, we restrict ourselves to~\eqref{eq:bestChoiceOfr1r01r12r012}.
	\begin{definition} \label{def:genCasteljau}
		Given control points $b_0,b_1,b_2,b_3 \in \mcal$
		we use $\beta(\cdot;b_0,b_1,b_2,b_3)$ to denote the \emph{generalized de Casteljau curve} constructed in Proposition~\ref{prop:retractionDeCast} with the choice~\eqref{eq:bestChoiceOfr1r01r12r012} for $r_1,r_{01},r_{12},r_{012}$.
	\end{definition}

	\subsection{The retraction-based Hermite (RH) interpolation scheme}
	
	The generalized de Casteljau curve of Definition~\ref{def:genCasteljau} will now be used to solve the Hermite interpolation problem, Problem~\ref{problem:HermiteInterpolation}, by choosing suitable control points.
	\begin{proposition}\label{prop:hermiteRetrBuildingBlock}
		Given $(p_0,v_0), (p_1, v_1)\in T\mcal$, define 
		\begin{equation}
			p_0^+ = R_{p_0}\pa{\frac{1}{3}v_0},\quad p_1^- = R_{p_1}\pa{-\frac{1}{3}v_1}
		\end{equation}
		and let $\alpha(t)\equiv \alpha(t;p_0,v_0,p_1,v_1):=\beta(t;p_0,p_0^+,p_1^-,p_1)$ denote the generalized de Casteljau curve according to Definition~\ref{def:genCasteljau}. Then
		\begin{equation}
			\alpha(0) = p_0,\	\alpha(1) = p_1,\quad 
			\dot\alpha(0) = v_0,\ \dot\alpha(1) = v_1.
		\end{equation}
	\end{proposition}
	\begin{proof}
		The result follows from combining Propositions~\ref{prop:generalizedDeCast} and~\ref{prop:retractionDeCast}:
		\begin{align}
			\alpha(0) &= \beta(0;p_0,p_0^+, p_1^-,p_1) = p_0,\ \alpha(1) = \beta(1;p_0,p_0^+, p_1^-,p_1) = p_1,\\
			\dot\alpha(0) &= \dot\beta(0;p_0,p_0^+, p_1^-,p_1) = 3 R_{p_0}^{-1}(p_0^+) = 3 R_{p_0}^{-1}\pa{R_{p_0}\pa{\frac{1}{3}v_0}} = v_0,\\
			\dot\alpha(1) &= \dot\beta(1;p_0,p_0^+, p_1^-,p_1) = -3 R_{p_1}^{-1}(p_1^-) = -3 R_{p_1}^{-1}\pa{R_{p_1}\pa{-\frac{1}{3}v_1}} = v_1.
		\end{align}  \qed 
	\end{proof}
	
	As an immediate consequence of Proposition~\ref{prop:hermiteRetrBuildingBlock}, the following corollary shows how $\alpha$ is used piecewise to define the \emph{retraction-based Hermite (RH) interpolant} $H$ that addresses Problem~\ref{problem:HermiteInterpolation}.
	\begin{corollary}\label{cor:hermiteRetrFullSolution}
		Letting $h_i := t_{i+1}-t_i$ for $i = 0,\dots,N-1$, the manifold curve $H:\pac{t_0,t_N}\to \mcal$ defined piecewise by 
		\begin{equation}
			\label{eq:fullPiecewiseInterpolant}
			H(t)\restr{\pac{t_{i}, t_{i+1}}} = \alpha\pa{\frac{t-t_i}{h_i};p_i, h_iv_i, p_{i+1}, h_iv_{i+1}},\quad i =0,\dots,N-1,
		\end{equation}
		is a solution to Problem~\ref{problem:HermiteInterpolation}.
	\end{corollary}
	
	Algorithms~\ref{alg:offlinePhase} and~\ref{alg:onlinePhase} summarize the construction of $H$ and its evaluation, respectively. We separate the computations needed for evaluating the RH interpolant (online phase) from those that can be precomputed, stored and used in every evaluation (offline phase). 
	
	\begin{algorithm}[ht]
		\caption{Offline Phase (precompute quantities defining the RH interpolant)}\label{alg:offlinePhase}		
		\textbf{Input:} Tangent bundle data points $\paa{(p_i,v_i)}_{i = 0}^N\in T\mcal$, $t_0<t_1<\dots<t_N$.
		\begin{algorithmic}[1]
			\For{$i = 0,\dots,N-1$}
			\State $h_i = t_{i+1} - t_i$;
			\State $p_i^+ = R_{p_i}\pa{\frac{1}{3}h_iv_i}$;
			\State $p_{i+1}^- = R_{p_{i+1}}\pa{-\frac{1}{3}h_{i}v_{i+1}}$;
			\State $q_i = R_{p_i^+}\pa{\frac{1}{2} R_{p_i^+}^{-1}\pa{p_{i+1}^-}}$; \Comment{Anchor for the middle segment $\beta_1$.}
			\State $w_i^{+} = R_{q_i}^{-1}\pa{p_i^+}$; \Comment{Tangent vector from $q_i$ to $p_i^+$}
			\State $w_{i+1}^- = R_{q_i}^{-1}\pa{p_{i+1}^-}$;\Comment{Tangent vector from $q_i$ to $p_{i+1}^-$}
			\EndFor\\
			\Return : $\paa{q_i, w_i^+,w_{i+1}^- }_{i=0}^{N-1}$;
		\end{algorithmic}
	\end{algorithm}
	
	\begin{figure}[ht]
		\centering
		\includegraphics[width=0.8\linewidth]{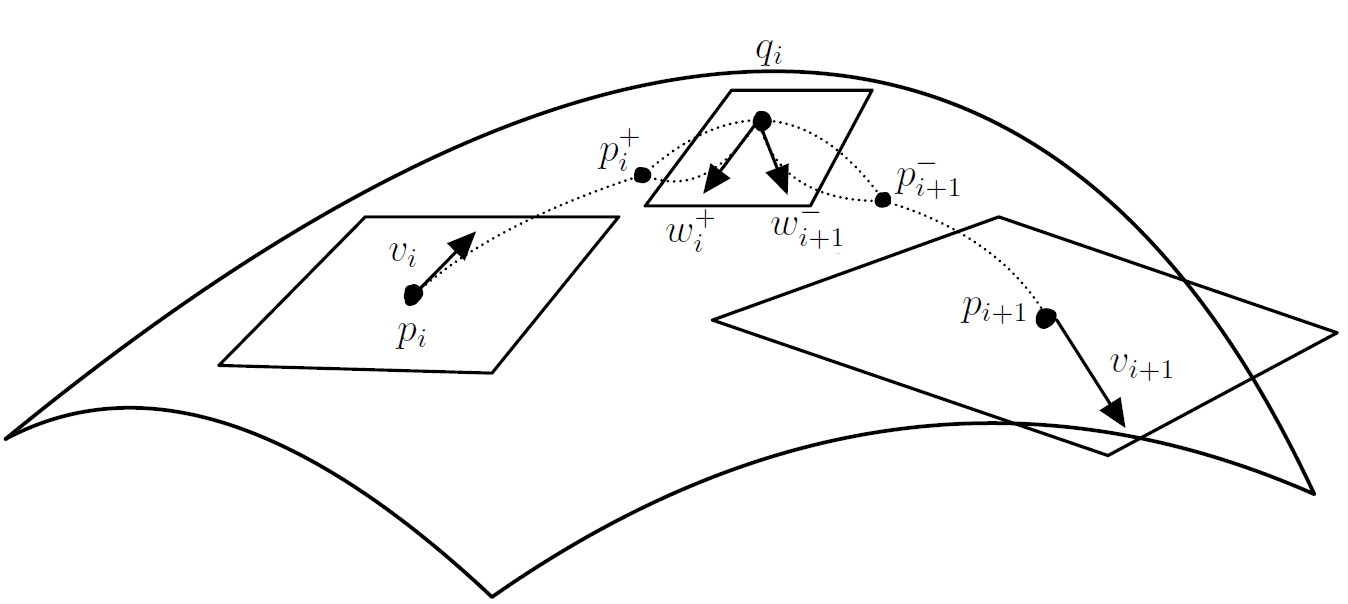}
		\caption{Illustration of offline computations performed by Algorithm~\ref{alg:offlinePhase}. Parallelograms indicate tangent spaces whereas dotted lines indicate retraction curves.}
		\label{fig:offlineCosts}
	\end{figure}
	
	\begin{algorithm}[ht]
		\caption{Online Phase (evaluation of RH interpolant at $t$)}	\label{alg:onlinePhase}		
		\textbf{Input: } $t\in\pac{t_0,t_N}$, $\paa{p_i, v_i, t_i}_{i =0}^N$, $\paa{q_i, w_i^+,w_{i+1}^- }_{i=0}^{N-1}$;
		\begin{algorithmic}[1]
			\State Find largest $i\in\paa{0,\dots,N-1}$ such that $t_i \leq t$;
			\State $h_i = t_{i+1}-t_i$;
			\State $\tau = \frac{t-t_i}{h_i}$;
			\State $\beta_0 = R_{p_{i}}\pa{\frac{\tau}{3}h_iv_i}$;\Comment{$1\times R$}
			\State $\beta_1 = R_{q_{i}}\pa{(1-\tau) w_i^+ + \tau w_{i+1}^-}$; \Comment{$1\times R$}
			\State $\beta_2 = R_{p_{i+1}}\pa{-\frac{\tau}{3}h_{i}v_{i+1}}$; \Comment{$1\times R$}
			\State $\beta_{01} = c_{0}\pa{\tau;\beta_0,\beta_1}$; \Comment{$1\times R + 1\times R^{-1}$}
			\State $\beta_{12} = c_{1}\pa{\tau;\beta_1,\beta_2}$; \Comment{$1\times R + 1\times R^{-1}$}\\
			\Return $\beta = c_\tau\pa{\tau;\beta_{01},\beta_{12}}$;  \Comment{$2\times R + 3\times R^{-1}$}
		\end{algorithmic}
	\end{algorithm}
	
	\section{Retraction-convex sets} \label{s:retractionConvexSets}
	
	The notion of an $r$-endpoint retraction curve $c_r$ from Definition~\ref{def:endpointRetr}
	is the central ingredient of our RH interpolation. To guarantee that the (inverse) retractions involved in $c_r$ and, in turn, $c_r$ itself are well-defined, we introduce the notion of a retraction-convex set.
	
	The definition of retraction-convex set is made possible by the local rigidity property of retractions ensuring the existence of their local inverse. The following property~\cite[Corollary 10.27]{boumalBook} conveniently characterizes the subset of the retraction's domain over which the inverse retraction can be defined.
	\begin{proposition}\label{prop:retractionDiffeom}
		For any retraction $R$ there exists a continuous function $\Delta:\mcal \to \papc{0,\infty}$ defining a subset $\tcal := \paa{(x,v)\in T\mcal :  \|v\|_x<\Delta(x)}$ of the tangent bundle such that 
		\begin{equation}
			\map{E}{\tcal \subset T\mcal}{\mcal\times\mcal}{(x,v)}{(x,R_x(v))}
		\end{equation}
		is a diffeomorphism.
	\end{proposition} 
	
	Accordingly, if $v\in T_x\mcal$ is such that $\norm{v}_x<\Delta(x)$, then the retraction is well-defined. Likewise, whenever $x,y\in \mcal$ are such that $(x,y)\in E(\tcal)$, we say that the inverse retraction $R_x^{-1}(y)$ is well-defined and is equal to the unique $v\in T_x\mcal$ such that $E(x,v) = (x,y)$. Moreover, with this definition the inverse retraction is smooth both in $x$ and $y$.

	\begin{definition}\label{def:retrConvexSet}
		A set $\ucal\subset\mcal$ is said to be \emph{retraction-convex} if for any $x,y,z\in\ucal$ 
		\begin{enumerate}[label=(\roman*)]
			\item the inverse retraction $R_x^{-1}(y)$ is well-defined, \label{item:invRetrIsDef}
			\item $R_x((1-\tau)R_x^{-1}(y) + \tau R_x^{-1}(z))$ is well-defined for every $\tau\in\pac{0,1}$ and belongs to $\ucal$, \label{item:endPointRetrIsInSet}
		\end{enumerate}
	\end{definition}
	An equivalent characterization of retraction-convexity is that the inverse retraction $R_x^{-1}$ is well-defined on $\ucal$ and the image of $\ucal$ through the inverse retraction is a convex subset of $T_x\mcal$ for every $x\in\ucal$. Retraction-convex sets generalize geodesically convex sets \cite[\S 6]{leeRiemManifs2}: instead of the geodesic it is the endpoint retraction curve between two points that remains in the set.

	\begin{proposition}\label{prop:rEndPointRetrCurveWellPosedness}
		Let $\ucal$ be retraction-convex and $x,y\in \ucal$. Then the $r$-endpoint retraction curve $t\mapsto c_r(t;x,y)$ is well-defined for every $t,r\in\pac{0,1}$.
	\end{proposition}
	\begin{proof}
		By Definition~\ref{def:retrConvexSet}~\ref{item:invRetrIsDef}, $R_x^{-1}(y)$ is well-defined. Now using Definition~\ref{def:retrConvexSet}~\ref{item:endPointRetrIsInSet} with $z = x$ and $\tau = 1-r$, we find that $q(r) = R_x(r R_x^{-1}(y))$ is well-defined and belongs to $\ucal$. Therefore, the inverse retractions $R_{q(r)}^{-1}(x)$ and $R_{q(r)}^{-1}(y)$ are also well-defined and the proof is completed by applying once again Definition~\ref{def:retrConvexSet}~\ref{item:endPointRetrIsInSet}. \qed
	\end{proof}

	\begin{proposition}\label{prop:generalizedDeCastRetrCurveWellPosedness}
		For a retraction-convex set $\ucal$ consider control points $b_0, b_1, b_2, b_3 \in \ucal$.
		Then the generalized de Casteljau curve $t \mapsto \beta(t;b_0,b_1,b_2,b_3)$ constructed in Definition~\ref{def:genCasteljau} is well defined for every $t\in\pac{0,1}$.
	\end{proposition}
	\begin{proof}
		We apply recursively the result of Proposition~\ref{prop:rEndPointRetrCurveWellPosedness}. Since $b_0,b_1,b_2,b_3\in\ucal$, the curves $\beta_0,\beta_1,\beta_2$ are well-defined and their image is entirely contained in $\ucal$. In turn, this implies that for every $t\in\pac{0,1}$ the curves $\beta_{01}(\cdot,t)$ and $\beta_{12}(\cdot, t)$ are well-defined and their image is entirely contained in $\ucal$. Finally, it follows that for every $s,t\in\pac{0,1}^2$ the curve $\beta_{012}(\cdot, s,t)$ is well-defined. \qed
	\end{proof}
	Note that this proof is valid for any generalized de Casteljau curve constructed as in Proposition~\ref{prop:retractionDeCast}.
	
	\subsection{Existence of retraction-convex sets}
	
	In the special case that the retraction is the exponential map, retraction-convexity coincides with geodesic convexity. The existence of geodesically convex neighborhoods around any manifold point is shown for instance by do Carmo \cite[Proposition 4.2]{doCarmo}.\footnote{Note that~\cite{doCarmo} uses the term strongly convex set instead of geodesically convex set.} The proof is based on the Gauss lemma~\cite[Lemma 3.5]{doCarmo}, a result from Riemannian geometry that crucially relies on the exponential map which does not extend to general retraction curves. Instead, our proof will only use properties of the manifold distance function, sharing common features with the proof \cite[Theorem 3.3.7]{doCarmo} that shows the existence of totally retractive neighborhoods as defined in~\cite[Section 3.3]{symmetricRankOneTR}.
	
	\begin{theorem}\label{theo:retrConvexExistence}
		For every $x\in\mcal$ there exists $\bar \rho > 0$ such that for any $\rho < \bar\rho$, the manifold ball $B(x,\rho)$ defined in~\eqref{eq:manifoldball} is retraction-convex. 
	\end{theorem}
	Given a $\bar\rho>0$ satisfying Theorem~\ref{theo:retrConvexExistence}, then the property is also satisfied by $\eta \bar\rho$, for any $\eta\in\pa{0,1}$. Thus, to fix a single value for each $x\in\mcal$, we denote 
	\begin{equation}\label{eq:rhoBarOfx}
		\bar\rho(x) := \sup\paa{\bar\rho>0 : B(x,\rho)\text{ is retraction-convex } \forall \rho<\bar\rho}.
	\end{equation} 
	
	The proof of Theorem~\ref{theo:retrConvexExistence} uses the following lemma.
	\begin{lemma}\label{lem:secondDerivSquareDistanceFunction}
		For every $x\in\mcal$ and any $v\in T_x\mcal$ we have
		\begin{equation}\label{eq:secondDerivativeOfRetractionDistance}
			\ddt{^2}{t^2} d(x,R_x(tv))^2\restr{t = 0} = 2\norm{v}_x^2.
		\end{equation}
	\end{lemma}
	\begin{proof}
		Setting $c(t):=R_x(tv)$, it follows
		from~\cite[Proposition 6.11]{leeRiemManifs2} that
		\begin{equation}
			d(x,c(t))^2 = \norm{\Log_x(c(t))}_x^2
		\end{equation}
		for sufficiently small $|t|$. This allows us to explicitly compute 
		\begin{align}
			\ddt{}{t}d(x,c(t))^2 = 2\scalp{\D \Log_x(c(t))[\dot c(t)]}{\Log_x(c(t))}_x,
		\end{align}
		and 
		\begin{equation}
			\begin{aligned}
				\ddt{^2}{t^2}d(x,c(t))^2 = 2 \langle & \D^2\Log_x(c(t))\pac{\dot c(t),\dot c(t)} +  \D\Log_x(c(t))[\ddot c(t)], \Log_x(c(t))\rangle_x\\
				+ &2\norm{\D\Log_x(c(t))[\dot c(t)]}_x^2.
			\end{aligned}
		\end{equation}
		Using that $c(0) = x$, the first term vanishes at $t = 0$ because $\Log_x(c(0)) = \Log_x(x) = 0$. Using
		local rigidity, $\D\Log_x(x) = I_{T_x\mcal}$,
		%(which follows from the of the Riemannian exponential map~\cite[Proposition 5.19]{leeRiemManifs2})
		and $\dot c(0) = v$ completes the proof:
		$
		\ddt{^2}{t^2} d(x,c(t))^2\restr{t = 0} = 2\norm{\D\Log_x(x)[v]}_x^2 = 2\norm{v}_x^2.
		$ \qed
	\end{proof}
	
	\begin{proof}(of Theorem~\ref{theo:retrConvexExistence})
		Given $x\in\mcal$, consider the function $f:T\mcal\to \Rbb$
		\begin{equation*}
			f(y,v) = d(x, R_y(v))^2.
		\end{equation*}
		The squared Riemannian distance function is smooth on a neighborhood of the diagonal of $\mcal\times\mcal$~\cite[Lemma 6.8]{leeRiemManifs2} and by definition the retraction $R$ is smooth on its domain (a neighborhood of the zero section of $T\mcal$). Therefore the function $f$ is smooth on an open neighborhood of $(x,0)$ in $T\mcal$.

		We will first establish a (local) convexity property of $f$.
		Using $\nabla_T^2f(x,w)$ to denote the (Euclidean) Hessian of $f(x,\cdot)$ at $w \in T_x\mcal$, the result of Lemma~\ref{lem:secondDerivSquareDistanceFunction} can be rephrased as
		\begin{equation}\label{eq:hessOfF}
			\scalp{\nabla_T^2f(x,0)\pac{v}}{v}_x = 2\norm{v}_x^2.
		\end{equation}
		In particular, this shows that $\nabla_T^2f(x,0)$ is positive definite. 

		By smoothness, there exists 
		a neighborhood $V\subset T\mcal$ containing $(x,0)$ such that 
		$\nabla_T^2f(y,v)$ remains positive definite at every $(y,v)\in V$.	

		For an arbitrary subset $S\subseteq T\mcal$, we let 
		\begin{equation}
			\begin{aligned}
				\piOnM(S) :=& \paa{y\in\mcal\,:\, \exists\, v\in T_y\mcal \text{ such that } (y,v)\in S },\\
				\label{eq:piOnT}\piOnT{y}(S) :=& \paa{v\in T_y\mcal\,:\, (y,v)\in S }.		
			\end{aligned}
		\end{equation}
		
		For every $y\in\piOnM(V)$ the function $f(y,\cdot)$ is convex on any convex subset of $\piOnT{y}(V)$.
		
		The set $V$ is an open set in the atlas topology of $T\mcal$, associated with the natural smooth structure of $T\mcal$~\cite[Lemma 4.1]{leeSmoothManifs1}. This means that for any local chart $\pa{\ucal, \phi}$ of $\mcal$, the set $\psi_\phi\pa{\piOnM^{-1}(\ucal)\cap V}$ is an open subset of $\Rbb^{2D}$, where 
		\begin{equation}
			\piOnM^{-1}(\ucal) = \paa{(y,v)\in T\mcal\,:\, y\in\ucal,\, v\in T_y\mcal} %= \underset{y\in \ucal}{\coprod} T_y\mcal
		\end{equation}  
		and 
		\begin{equation}
			\begin{aligned}
				\map{\psi_\phi}{\piOnM^{-1}(\ucal)\cap V}{\Rbb^{2D}}{(y,v)}{(\phi(y),\lambda)}
			\end{aligned}
		\end{equation}
		with $\lambda\in\Rbb^D$ being the coordinates of $v$ in the basis of $T_y\mcal$ formed by the partial derivatives of the inverse local charts, that is $v = \sum_{i = 1}^{D} \lambda_i \partial_i\phi^{-1}(\phi(y))$.
		
		Now consider a local chart such that $x\in\ucal$. Because $\psi_\phi\pa{\piOnM^{-1}(\ucal)\cap V}$ is an open subset of $\Rbb^{2D}$ that contains $\pa{\phi(x),0}$, there exist $\varepsilon_{\mcal}>0$ and $\varepsilon_T>0$ such that 
		\begin{equation}
			B_D(\phi(x),\varepsilon_{\mcal})\times B_D(0,\varepsilon_T)\subseteq \psi_\phi\pa{\piOnM^{-1}(\ucal)\cap V},
		\end{equation}
		where $B_D(\cdot,\cdot)$ denotes an open ball (in the Euclidean norm) of $\Rbb^D$. By continuity of $\psi_\phi$, its preimage is an open subset of $V$ that contains $(x,0)$:
		\begin{equation}
			\begin{aligned}
				W :=& \psi_\phi^{-1}(B_D(\phi(x),\varepsilon_{\mcal})\times B_D(0,\varepsilon_T))\\
				=& \paa{(y,v)\in T\mcal: y\in \phi^{-1}\pa{B_D(\phi(x),\varepsilon_{\mcal})},\, v = \sum_{i = 1}^D\lambda_i\partial_i\phi^{-1}(\phi(y)), \,\norm{\lambda}_2 < \varepsilon_T}.
			\end{aligned}
		\end{equation}
		From Proposition~\ref{prop:retractionDiffeom}, we have that $E$ is a diffeomorphism on $W\cap \tcal$ (which is open in $T\mcal$) and so the set $E(W\cap\tcal)$ is an open neighborhood of $(x,x)$. Thus, the constant 
		\begin{equation} \label{eq:barrho}
			\bar\rho := \sup \paa{\rho \geq 0: B(x,\rho)\times B(x,\rho)\subseteq E(W\cap\tcal)}.
		\end{equation}
		is strictly positive.
		
		The statement of the theorem is proven by showing that the set $B(x,\rho)$ is retraction-convex
		for any $\rho < \bar\rho$. Definition~\ref{def:retrConvexSet}~(i) follows from the invertibility of $E$ on $B(x,\rho)\times B(x,\rho)$: For any $y,z\in B(x,\rho)$ there exists a unique $v\in T_y\mcal$ such that $E(y,v)=(y,z)$ and, hence, $v = R^{-1}_y(z)$ is well-defined. 
		To establish Definition~\ref{def:retrConvexSet}~(ii), consider arbitrary $w,y,z\in B(x,\rho)$. We first note that the set 
		\begin{equation}
			\piOnT{y}(W\cap\tcal) = \paa{v\in T_y\mcal: v = \sum_{i = 1}^D\lambda_i\partial_i\phi^{-1}(\phi(y)), \,\norm{\lambda}_2 < \varepsilon_T, \norm{v}_y<\Delta(y)}
		\end{equation} 
		is convex (as the intersection of two convex sets) and in the domain of $R_y$. Both $R_y^{-1}(w)$ and $R_y^{-1}(z)$ are contained in $\piOnT{y}(W\cap\tcal)$ and hence the same holds for their convex linear combination. The convexity of $f(y,\cdot)$ on $\piOnT{y}(W\cap\tcal)$, a convex subset of $\piOnT{y}(V)$, implies for every $t\in\pac{0,1}$ that
		\begin{align}
			d\big(x, R_y((1-t)R_y^{-1}(w) + t R_y^{-1}(z))\big)^2 &=
			f\big(y, (1-t)R_y^{-1}(w) + t R_y^{-1}(z)\big) \\
			&\leq (1-t)f\big(y, R_y^{-1}(w) \big) + t f\big(y, R_y^{-1}(z)\big)\\
			&= (1-t) d(x,w)^2 + t d(x,z)^2 < \rho^2,
		\end{align}
		which proves the retraction-convexity of $B(x,\rho)$. \qed
	\end{proof} 
	
	\section{Analysis of RH interpolation} \label{s:analysis}
	
	For the purpose of deriving qualitative and asymptotic properties of RH interpolation, we suppose that the interpolation data are samples of a continuously differentiable manifold curve ${\gamma:\pac{0,1}\to\mcal}$, that is
	\begin{equation}\label{eq:equispacedInterpProblem}
		p_i = \gamma(t_i),\quad v_i =\dot\gamma(t_i), \quad \forall i = 0,\dots,N,
	\end{equation} 
	for some $0=t_0<t_1<\dots<t_N = 1$. Because of its piecewise definition (see Corollary~\ref{cor:hermiteRetrFullSolution}), it is sufficient to consider the RH interpolant on a single subinterval, i.e. the manifold curve $H_h:\pac{t,t+h}\to \mcal$ defined by
	\begin{equation} \label{eq:hh}
		H_h(\tau) = \alpha\pa{\frac{\tau-t}{h};\gamma(t),h\dot\gamma(t), \gamma(t+h), h\dot\gamma(t+h)}
	\end{equation}
	for sufficiently small $h > 0$ and $t\in\pac{0,1-h}$, satisfying $H_h(t) = \gamma(t)$, $H_h(t + h) = \gamma(t+h)$,
	$\dot H_h(t) = \dot\gamma(t)$, and $\dot H_h(t + h) = \dot \gamma(t+h)$.
	Our results for a single interval apply to the piecewise solution of~\eqref{eq:equispacedInterpProblem} by letting $h=\max_{i = 0,\dots,N-1}  t_{i+1} - t_i$. 
	
	\subsection{Lipschitz continuity of the retraction}\label{ss:lipschitzContinuityOfRetraction}
	
	Lemma~\ref{lem:lipschRetrAndInvRetr} below states the Lipschitz continuity of (inverse) retractions. A very similar statement can be found in~\cite[Lemma 6]{ringWirth}, under a local equicontinuity assumption of the retraction derivatives. Our proof, reported in Appendix~\ref{app:lipschitzContinuityOfRetraction} together with the proof of Corollary~\ref{cor:lipschitzRetrOnCompact}, leverages  Proposition~\ref{prop:retractionDiffeom} to avoid this assumption.

	\begin{lemma}\label{lem:lipschRetrAndInvRetr}
		For every $x\in\mcal$, there exist $L_R(x)>0$ and $M_R(x)>0$ such that
		\begin{enumerate}[label=(\roman*)]
			\item \label{cor:lipschRetr} for any $u,v\in\paa{w\in T_x\mcal : \norm{w}_x < \Delta(x) / 3}$ we have
			\begin{equation}\label{eq:retrLipschitzCont}
				d(R_x(u), R_x(v))\leq L_R(x) \|u-v\|_x,
			\end{equation}
			\item \label{cor:lipschInvRetr} for any $y,z\in \paa{R_x(w) \in \mcal: \norm{w}_x < \Delta(x) / 3}$ we have
			\begin{equation}\label{eq:invRetrLipschitzCont}
				\norm{R_x^{-1}(y)- R_x^{-1}(z)}_x\leq M_R(x) d(y,z),
			\end{equation} 
		\end{enumerate}
		with $\Delta(x)>0$ defined in Proposition~\ref{prop:retractionDiffeom}. 
	\end{lemma}
	\begin{corollary}\label{cor:lipschitzRetrOnCompact}
		For every compact set $K\subset \mcal$, inequalities \eqref{eq:retrLipschitzCont} and \eqref{eq:invRetrLipschitzCont} hold for every $x\in K$ with finite strictly positive constants $L_R(K)$ and $M_R(K)$ depending only on $K$.
	\end{corollary}
	\begin{proof}
		This is a consequence of the more general Proposition~\ref{prop:lipschitzRetrOnCompactGeneral}, see Appendix~\ref{app:lipschitzContinuityOfRetraction}. \qed
	\end{proof}

	\subsection{Well-posedness of RH interpolation} \label{ss:wellPosedness}
	Throughout the two upcoming sections we make the assumption that the function ${x\in\mcal\to\bar\rho(x)}$ defined by $\eqref{eq:rhoBarOfx}$ is lower-bounded on the image of the curve $\gamma$ by a strictly positive constant that we denote
	\begin{equation}
		\label{eq:rhoMin}
		\rho_{\min} := \inf_{\tau\in\pac{0,1}}{\bar\rho(\gamma(\tau))}>0.
	\end{equation}

	Since retractions are smooth, intuition suggests that the function $\bar\rho$ should be continuous, thereby guaranteeing the validity of this technical assumption.

	By Proposition~\ref{prop:generalizedDeCastRetrCurveWellPosedness}, the RH interpolation scheme is well-defined provided that the control points belong to a retraction-convex set. The following result is established that this indeed the case for sufficiently small $h$.

	\begin{proposition}\label{prop:wellDefinitessOfRHOfCurve} There exists a constant $ h_1 >0$ depending on the curve $\gamma$ and on the retraction such that for any $0<h< h_1$ the RH interpolant $H_h$ defined in~\eqref{eq:hh} is well-posed for every $t\in\pac{0,1-h}$.
	\end{proposition}
	\begin{proof}
		By Proposition~\ref{prop:hermiteRetrBuildingBlock} and Corollary~\ref{cor:hermiteRetrFullSolution},
		\[
		H_h(\tau) 
		= \beta\pa{\frac{\tau-t}{h};p_0(t),p_0^+(t,h),p_1^-(t,h),p_1(t,h)},
		\]
		with the control points
		\begin{align}
			p_0(t) &:=  \gamma(t),\quad &&p_1(t,h), := \gamma(t+h), \\
			p_0^+(t,h)&:= R_{\gamma(t)}\pa{{h\dot\gamma(t)}/{3}},\quad &&p_1^-(t,h) := R_{\gamma(t+h)}\pa{-{h\dot\gamma(t+h)}/{3}}.
		\end{align}
		
		Denoting by $L_\gamma$ a Lipschitz constant of the curve $\gamma$ on $\pac{0,1}$ we have
		\begin{equation*}		
			d(p_0(t),p_1(t,h)) = d(\gamma(t),\gamma(t+h)) \leq L_\gamma h,
		\end{equation*}	
		The function $\Delta(\cdot)$ defined in Proposition~\ref{prop:retractionDiffeom}, is continuous and strictly positive so the quantity
		$\Delta_{\min} = \min_{\tau\in\pac{0,1}}\Delta(\gamma(\tau))$ is strictly positive. If $h<\frac{\Delta_{\min}}{L_\gamma}$, then 
		\begin{align*}
			\norm{{h\dot\gamma(t)}/3}_{\gamma(t)} < {\Delta(\gamma(t))}/3, \quad \norm{{h\dot\gamma(t+h)}/3}_{\gamma(t+h)} < {\Delta(\gamma(t+h))}/3.
		\end{align*}	
		This allows us to invoke  Lemma~\ref{lem:lipschRetrAndInvRetr}-\ref{cor:lipschRetr} with the constant $L_R(\gamma) := \sup_{t\in\pac{0,1}}\paa{ L_R\pa{\gamma(t)}}$. This constant is finite by Corollary~\ref{cor:lipschitzRetrOnCompact}.
		We have 
		\begin{align*}
			d(p_0(t),p_0^+(t,h)) &\leq hL_R \norm{\dot\gamma(t) / 3}_{\gamma(t)} \leq h L_RL \gamma/3,\\
			d(p_0(t),p_1^-(t,h)) &\leq d(p_0(t),p_1(t,h)) + d(p_1(t,h),p_1^-(t,h)) \\
			&\leq L_\gamma h +   hL_R{\norm{\dot\gamma(t+h) / 3}_{\gamma(t+h)}} \leq L_\gamma(1 + L_R/3) h.		
		\end{align*} 
		Hence, if $h<\frac{\Delta_{\min}}{L_\gamma}$, then all control point of $H_h$ are contained in $B(\gamma(t),Qh)$ with
		\begin{equation*}
			Q := \max\paa{L_\gamma,\,L_RL_\gamma/3 ,\, L_\gamma(1 + L_R/3)} = L_\gamma(1 + L_R/3).
		\end{equation*}
		Therefore by taking $h <  h_1 := \min\paa{\frac{\Delta_{\min}}{L_\gamma},\frac{\rho_{\min}}{Q}}$ all the control points of $H_h$ are contained in the retraction-convex set $B(\gamma(t), \rho_{\min})$. This implies the curve $H_h$ is well-defined. \qed
	\end{proof}
	
	\subsection{Interpolation error}
	In the Euclidean setting, asymptotic convergence rates of the maximum interpolation error are classic results of numerical analysis. For the case of piecewise cubic Hermite interpolation the error can be shown to converge as $O(h^4)$ \cite[Section 8.4]{quarteroniBook}, where $h$ is the largest sampling step size. In this section, we generalize this result to the RH interpolation scheme. 
	\begin{theorem} \label{theo:convergenceOrder}
		Let $\gamma\in C^4(\pac{0,1})$ and consider $H_h$, the RH interpolant of $\gamma$ on a subinterval $[t,t+h]$ as defined in \eqref{eq:hh}. Assume that $\gamma$ satisfies \eqref{eq:rhoMin} and that for $k=2,3,4$ there exist constants $L_{RH}^{(k)}>0$ and $h_2>0$ such that for every $0<h< h_2 $ and any $t\in\pac{0,1-h}$  it holds that 
		\begin{equation}\label{eq:boundednessOfKthDerivOfHh}
			\supOn{\tau\in\pac{t,t+h}}\norm{\Ddt{^{k}H_h(\tau)}{\tau^k}}_{H_h(\tau)} < L_{RH}^{(k)}, 
		\end{equation} 
		where $\Ddt{^k}{\tau^k}$ denotes the order-$k$ covariant derivative along a curve~\cite[Sec. 10.7]{boumalBook}. Then, there exist constants $\kappa>0$ and $\bar h>0$ depending on the curve $\gamma$, the manifold and the retraction such that for any $0<h< \bar h$ and any $t\in\pac{0,1-h}$, we have
		\begin{equation}
			\underset{\tau\in\pac{t,t+h}}{\max} d(\gamma(\tau), H_h(\tau)) < \revision{\kappa} h^4.
		\end{equation}
	\end{theorem}
	
	The above result is independent of $t\in\pac{0,1-h}$, hence we can bound the maximum interpolation error of the full piecewise RH interpolant of the curve $\gamma$ on the interval $\pac{0,1}$.
	\begin{corollary}
		Under the assumptions of Theorem~\ref{theo:convergenceOrder}, the piecewise RH interpolant $H$ of $\gamma$ defined by \eqref{eq:fullPiecewiseInterpolant} with data \eqref{eq:equispacedInterpProblem} verifies 
		\begin{equation*}
			\underset{\tau\in\pac{0,1}}{\max} d(\gamma(\tau), H(\tau)) < \revision{\kappa} h^4,
		\end{equation*}
		where $h=\max_{i = 0,\dots,N-1}  t_{i+1} - t_i$.
	\end{corollary}
	
	The proof of Theorem~\ref{theo:convergenceOrder} relies on a local linearization of the manifold: the interpolation curve $H_h$ and the curve $\gamma$ are represented in normal coordinates \cite[p. 132]{leeRiemManifs2} around a point of the curve $\gamma$ by applying the local inverse of the Riemannian exponential map. This is possible provided that the image of $\gamma$, the control points, and intermediate quantities involved in the procedure remain confined to a domain of invertibility of the  exponential map as $h\rightarrow 0$. The following Lemmas~\ref{lem:lipschitzContinuityIndepOfH} and~\ref{lem:representabilityInNormalCoordinates} show that these requirements are met when considering sufficiently small $h$, reflecting the asymptotic nature of Theorem~\ref{theo:convergenceOrder}.  In practice, our experiments suggest that the convergence order $4$ established by the theorem can be observed as soon as the interpolation curve is well-defined; see the numerical experiments Section~\ref{s:numericalExperiments}. This assumes that 
	the scheme verifies assumption~\eqref{eq:boundednessOfKthDerivOfHh}; see Section~\ref{ss:needForBoundedDerivatives} for the importance of this assumption.
	
	\begin{lemma}\label{lem:lipschitzContinuityIndepOfH}
		There exist a constants $L_{RH}>0$ and $h_3 > 0$ depending on $\gamma$ and on the retraction such that for every $0<h< h_3$ and any $t\in\pac{0,1-h}$
		\begin{equation}
			d(H_h(\tau_1),H_h(\tau_2)) \leq L_{RH} |\tau_1 - \tau_2|, \quad \forall \tau_1, \tau_2\in \pac{t,t+h}.
		\end{equation} 
	\end{lemma}
	
	A proof of Lemma~\ref{lem:lipschitzContinuityIndepOfH} can be found in Appendix~\ref{app:lipschitzContinuityIndepOfHProof}. Lipschitz continuity of $H_h$ and of $\gamma$ provides sufficient conditions to be able to represent the images of $H_h$ and $\gamma$ on the interval $\pac{t,t+h}$ in normal coordinates around any point $\gamma(s)$, $s\in \pac{t,t+h}$.
	\begin{lemma}\label{lem:representabilityInNormalCoordinates}
		Denote $r_{\min} := \min_{\tau\in\pac{0,1}}\operatorname{inj}(\gamma(\tau))$, the minimum of the injectivity radius of the Riemannian exponential map along the curve. There exists $ h_4 > 0$ such that for every $0<h< h_4$ and any $t\in\pac{0,1-h}$ we have 
		\begin{equation*}
			\begin{aligned}
				d(\gamma(s), \gamma(\tau)) &< r_{\min},\\
				d(\gamma(s), H_h(\tau)) &< r_{\min},
			\end{aligned} \quad \forall\,s,\tau \in \pac{t,t+h}.
		\end{equation*}
	\end{lemma}
	\begin{proof}
		The constant $r_{\min}$ is strictly positive by the continuity of the injectivity radius function~\cite[Proposition 10.24]{boumalBook} and compactness of the image of the curve. We take any $0 < h < h_4 := \min\paa{\frac{r_{\min}}{2L_\gamma}, h_3, \frac{r_{\min}}{2L_{RH}}}$ and consider an arbitrary $t\in\pac{0,1-h}$. First of all, for any $s,\tau\in\pac{t,t+h}$ we have 
		\begin{equation}
			d(\gamma(s),\gamma(\tau)) \leq L_\gamma h < \frac{r_{\min}}{2} < r_{\min},
		\end{equation}
		Furthermore, the requirement $h< h_3$ guarantees that $H_h$ is well-defined and, by Lemma~\ref{lem:lipschitzContinuityIndepOfH}, that $H_h$ is $L_{RH}$-Lipschitz continuous. Thus
		\begin{align*}
			d(\gamma(s),H_h(\tau)) \leq d(\gamma(s),\gamma(t)) + d(\gamma(t),H_h(\tau))&< \frac{r_{\min}}{2} + d(H_h(t),H_h(\tau))\\ &< \frac{r_{\min}}{2} + L_{RH} h\\ &< \frac{r_{\min}}{2} + \frac{r_{\min}}{2} = r_{\min}.
		\end{align*} \qed
	\end{proof}
	
	\begin{proof}(of Theorem~\ref{theo:convergenceOrder})
		Take any $0<h<h_4$, $t\in\pac{0,1-h}$ and $s\in\pac{t,t+h}$. By Lemma~\ref{lem:representabilityInNormalCoordinates}, we can express $\gamma$ and $H_h$ in normal coordinates. Define 
		\begin{equation}\label{eq:normalCoordinateCurves}
			\begin{aligned}
				\hat\gamma_s &= \varphi_s(\gamma),\\
				\hat H_s &= \varphi_s(H_h),
			\end{aligned}
		\end{equation}
		where, for any $p\in B(\gamma(s),\operatorname{inj}(\gamma))$, $\varphi_s(p) = V_{\gamma(s)}^{-1}\circ \Exp_{\gamma(s)}^{-1}(p)$ and $V_x:\Rbb^D\to T_x\mcal$ is the basis isomorphism associated to any orthonormal basis $\paa{b_i}_{i=1}^D$ of $T_x\mcal$, for any $x\in\mcal$ .
		
		The interpolation error in normal coordinates at $\gamma(s)$ is defined as $\hat E_s(\tau) = \hat\gamma_s(\tau) - \hat H_s(\tau)$ for every $\tau \in \pac{t,t+h}$ and by construction satisfies 
		%\begin{equation}\label{eq:interpCondInNormalCoords}
		\[
			\begin{aligned}
				\hat E_s(t) = 0,&& \quad && \hat E_s(t+h) = 0,\\
				\ddt{}{\tau}\hat  E_s(\tau)\restr{\tau = t} = 0,&& \quad && \ddt{}{\tau}\hat  E_s(\tau)\restr{\tau = t + h} = 0.
			\end{aligned}
		\]
		
		By applying iteratively Rolle's theorem on the components $\hat E_{s,i}$ for each $i = 1,\dots,D$, we can show there exists $\xi_i\in\pac{t,t+h}$ such that
		\[
			|\hat E_{s,i}(\tau)| \leq \frac{|\hat E_{s,i}^{(4)}(\xi_i)|}{4!} h^4,
		\]
		where the exponent in parentheses indicates the order of differentiation.
		Therefore, for any $s\in\pac{t,t+h}$ we can bound the norm of \revision{the} interpolation error as follows:
		\begin{equation*}
			\begin{aligned}
				\norm{\hat E_s(\tau)}_{2} &\leq \frac{h^4}{4!} \pa{\sum_{i=1}^D |\hat E_{s,i}^{(4)}(\xi_i)|^2}^{1/2}\\
				&\leq \frac{h^4}{4!} \sqrt{D} \max_{i = 1,\dots,D} |\hat E_{s,i}^{(4)}(\xi_i)|\\
				&\leq \frac{h^4}{4!} \sqrt{D} \max_{i = 1,\dots,D}\norm{\hat E_{s}^{(4)}(\xi_i)}_2\\
				&\leq   \frac{h^4}{4!} \sqrt{D} \max_{\xi\in \pac{t,t+h}} \norm{\hat E^{(4)}_s(\xi)}_2.
			\end{aligned}
		\end{equation*}
		An important property of normal coordinates is that radial directions map to length-minimizing geodesics, thus $d(\gamma(\tau), H(\tau)) = \norm{V_{\gamma(\tau)} \hat E_\tau(\tau)}_{\gamma(\tau)} = \norm{\hat E_\tau(\tau)}_2$. So we can say
		\begin{align*}
			d(\gamma(\tau), H(\tau)) &\leq  \frac{h^4}{4!} \sqrt{D} \max_{s,\xi\in \pac{t,t+h}} \norm{\hat E^{\revision(4)}_s(\xi)}_2\\
			&\leq \frac{h^4}{4!} \sqrt{D} \max_{s,\xi\in \pac{t,t+h}} \paa{\norm{\hat \gamma_s^{(4)}(\xi)}_2 + \norm{\hat H_s^{(4)}(\xi)}_2}
		\end{align*}
		
		By the smoothness of the exponential map and its inverse, applying the chain rule to~\eqref{eq:normalCoordinateCurves} gives an expression for $\hat\gamma_s^{(4)}$ and $\hat H_s^{(4)}$ respectively as the sum of contributions of the form 
		\begin{equation*}
			\D^{k}\Exp_{\gamma(s)}^{-1}(\gamma(\xi))\pac{\Ddt{^{i_1}\gamma(\xi)}{\xi^{\revision{i_1}}},\dots,\Ddt{^{\revision{i_k}}\gamma(\xi)}{\xi^{\revision{i_k}}}}
		\end{equation*}
		and 
		\begin{equation*}		\D^{k}\Exp_{\gamma(s)}^{-1}(H_h(\xi))\pac{\Ddt{^{i_1}H_h(\xi)}{\xi^{\revision{i_1}}},\dots,\Ddt{^{i_k}H_h(\xi)}{\xi^{i_k}}}
		\end{equation*}
		with $k\in\paa{1,2,3,4}$ and $i_j \geq 0$ such that $i_1+\dots + i_k = 4$ and where $\D^k\Exp^{-1}_{\gamma(s)}(\cdot)$ is the multilinear operator associated to the differential of order $k$. Then, the norm of each term is bounded by a constant independent of $h$ (the maximum of the multilinear operator norms on the curve image) multiplied by the Lipschitz constant of the curve $\gamma$, the suprema of the norms of its higher order covariant derivatives and, provided $h<h_2$, the constants given by assumption~\eqref{eq:boundednessOfKthDerivOfHh}. \revision{In summary, this} produces a constant $\kappa>0$ such that for any $h < \bar h :=\min\paa{h_2,h_4}$ and any $t\in\pac{0,1-h}$ we have 
		\begin{equation*}
			\underset{\tau\in\pac{t,t+h}}{\max} d(\gamma(\tau), H_h(\tau)) < \revision{\kappa} h^4.
		\end{equation*} \qed
	\end{proof} 
	
	\section{Numerical experiments} \label{s:numericalExperiments}
	The following section is dedicated to numerical experiments illustrating the RH interpolation method, summarized in Algorithms~\ref{alg:offlinePhase} and~\ref{alg:onlinePhase}. All experiments have been carried out in Matlab 2019b leveraging the differential geometry tools of the Manopt library \cite{manoptPaper} on a laptop computer with Intel i7 CPU (1.8GHz with single-thread mode) with 8GB of RAM, 1MB of L2 cache and 8MB of L3 cache.
	
	\subsection{Manifolds and retractions of interest}
	
	Let us briefly introduce the manifolds and retractions used in our experiments.
	
	\subsubsection{Stiefel manifold}
	The set of column orthogonal matrices of size $ n\times k$ is a manifold of dimension $nk-\frac{k(k+1)}{2}$ that is commonly known as the (compact) Stiefel manifold and denoted by
	\begin{equation}
		\stiefel{n}{k} := \paa{X\in\matr{n}{k}: X^\top X = I_{k\times k}}.
	\end{equation}
	We endow the manifold with the Riemannian submanifold structure, i.e. by inducing on each tangent space $T_X \stiefel{n}{k}$ the standard inner product on $\matr{n}{k}$. An exhaustive description of this structure can be found in \cite[\S 7.3]{boumalBook} or \cite[\S 2.2]{edelmanExpForStiefel}. 
	On the Stiefel manifold we are aware of two retractions for which the inverse retraction can be conveniently computed. 
	Given $X\in\stiefel{n}{k}$ and $V\in T_X\stiefel{n}{k}$ we define:
	\begin{itemize}
		\item the Q-factor retraction \cite[\S 7.3]{boumalBook}: \begin{equation*}
			R_X^{\mathrm{Q}}\revision{(V)} = \mathrm{qf}(X + V)
		\end{equation*}
		where $\mathrm{qf}(A) = Q \in \mathbb R^{n\times k}$ such that $A = QR$ is the unique (thin) QR decomposition of $A$. The uniqueness is guaranteed if $A$ has full rank (which is the case for $X + V$) and if we enforce the diagonal entries of $R$ to be positive. A procedure to compute the inverse Q-factor retraction is proposed in \cite[Algorithm 1]{stiefelInvRetr} and consists of solving a linear matrix equation for an upper triangular matrix. 
		\item the P-factor retraction \cite[\S 3.3]{projLikeRetractions}: \begin{equation*}
			R_X^{\mathrm{P}}\revision{(V)} = \mathrm{pf}(X + V)
		\end{equation*}
		where $\mathrm{pf}(A) = P$ such that $A = PS$ is the unique polar decomposition of the matrix A. The polar decomposition can be obtained from the SVD of the matrix. An algorithm to compute the inverse P-factor retraction is proposed in \cite[Algorithm 2]{stiefelInvRetr} and requires solving a Lyapunov matrix equation.
	\end{itemize} 
	
	\subsubsection{Fixed-rank manifold}
	The set of $m\times n$ matrices of rank $k\leq \min\paa{m,n}$ denoted by
	\begin{equation}
		\mcal_k := \paa{X \in \Rbb^{m\times n}: \rank{X} = k}
	\end{equation}
	is a manifold of dimension $(m+n-k)k$ embedded in $\Rbb^{m\times n}$ \cite[Example 8.14]{leeSmoothManifs1}.
	
	Among many equivalent parametrizations of this manifold, we choose to represent points of $\mcal_k$ with their economy-sized SVD. The manifold can be made an embedded Riemannian submanifold of $\matr{m}{n}$ by inducing on each tangent space the inner product of $\matr{m}{n}$. We refer the reader to \cite{bartOnMk}~\cite{bartUschmajewOnMk} for details on the representation and geometry of the manifold.
	
	Possibly the most frequently used retraction on the rank-$k$ matrix manifold is the $k$-truncated SVD retraction. However, we are not aware of any procedure allowing to compute efficiently the corresponding inverse retraction. As highlighted by Absil and Malick \cite{projLikeRetractions}, an alternative retraction for the fixed-rank matrix manifold is the so-called orthographic retraction. It consists of perturbing the point $X$ in the ambient space as $X+V\in\matr{m}{n}$ and projecting back onto the manifold along vectors from the normal space of the starting point, see Figure~\ref{fig:orthographicRetr}. Formally this reads:
	\begin{equation}
		R_X: V\in T_X\mcal_k\mapsto \underset{P\in(X+V+N_X\mcal_k)\cap\mcal_k}{\arg\min}\norm{X + V - P}_F^2.
	\end{equation}
	A closed-form expression for the solution of this optimization problem can be found in \cite[\S 3.2]{absilOseledets}. The major advantage of the orthographic retraction is that the inverse retraction is trivial. As illustrated by Figure~\ref{fig:orthographicRetr}, it is sufficient to project the ambient space difference onto the tangent space, i.e. ,
	\begin{equation}
		R_X^{-1}(Y) = \Pi_X(Y-X).
	\end{equation}
	
	\begin{figure}
		\centering
		\includegraphics[width=0.65\linewidth]{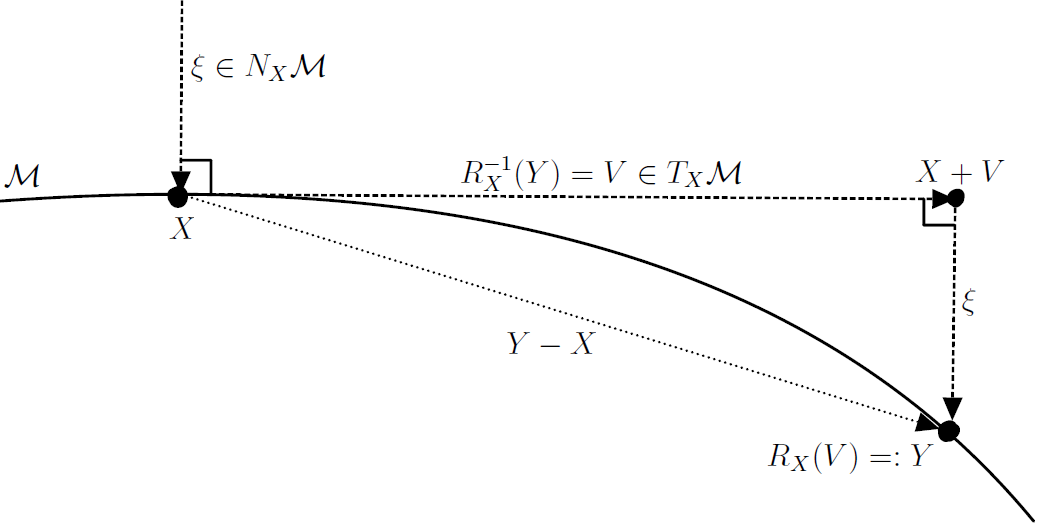}
		\caption{Orthographic retraction and its inverse.}
		\label{fig:orthographicRetr}
	\end{figure}
	
	\subsection{Academic examples}\label{ss:academicExamples}
	In the following sections, we illustrate the RH interpolation scheme on two problems arising from numerical linear algebra: the computation of a smooth QR decomposition and a smooth singular value decomposition for a given smooth matrix curve $t\in\pac{a,b}\to A(t) \in \matr{m}{n}$. Assuming for the moment that these smooth decomposition exist and can be computed, the experimental setup is the following. We sample the decomposition and its first order derivative at uniformly spaced location\revision{s} and interpolate this data with different manifold interpolation schemes. We then vary the sampling step size $h$ and measure the maximum interpolation error with respect to the original smooth decomposition. 
	
	\subsubsection{Comparing with other retraction-based schemes}\label{sss:otherSchemes}
	The RH method is compared with two \revision{other} interpolation schemes that only use retractions. First, the analogous of the piecewise linear interpolant is defined as $L(t)\restr{[t_i,t_{i+1}]} = c_0\pa{\frac{t-t_i}{h}, p_i, p_{i+1}}$, where $c_0$ is the endpoint retraction curve defined in Definition~\ref{def:endpointRetr}. We then consider a naive piecewise Hermite interpolant \revision{that} uses the same control points as the RH scheme but where only the endpoint curve $c_0$ is used as the building block for the generalized de Casteljau procedure. The resulting curve is then not expected to be continuously differentiable at the junctions.
	
	\subsubsection{Implementation details}
	
	%In the following experiments, the computation of manifold curve derivatives is achieved with a
	
	To measure the error of an approximation $t\to\tilde A(t)$ to a matrix manifold curve $A(t)$, we consider the pointwise errors 
	\begin{equation*}
		\varepsilon_{P}(t) = \norm{A(t) - \tilde A(t)}_F
		\quad \text{and} \quad
		\varepsilon_D(t) = \norm{\dot A(t) - \dot{\tilde {A}}(t)}_F.
	\end{equation*}
	For convenience and the purpose of these experiments, we compute all required derivatives via centered finite differences: 
	\begin{equation}\label{eq:finiteDifferenceOnManif}
		\dot A(t) \simeq \Pi_{A(t)}\pa{\frac{A(t + \Delta t) - A(t - \Delta t)}{2\Delta t}}, \quad \Delta t = 10^{-5}.
	\end{equation}
	Note that Theorem~\ref{theo:convergenceOrder} uses the Riemannian distance to measure the interpolation error while we measure the error with the ambient space distance. It can be shown that if an embedded manifold is endowed with the induced metric, the Euclidean distance is locally equivalent to the Riemannian distance, see e.g. \cite[Appendix A]{attali}.
	
	\subsubsection{Q-factor interpolation}
	For convenience, the example matrix curve we consider is the same as in \cite[\S 5.2]{zimmermann}. It consists of the matrix polynomial 
	\begin{equation}\label{eq:qfactorInterpolationInstance}
		Y(t) = Y_0 + t Y_1 + t^2 Y_2 + t^3 Y_3,\quad Y_i \in \matr{n}{k}, n = 500, k = 10, t\in\pac{-1.1,1.1},
	\end{equation} 
	where the entries of the matrices $Y_i$ are pseudo-randomly generated from uniform distributions on $\pac{0,1}$, $\pac{0,0.5}$, $\pac{0,0.5}$, $\pac{0,0.2}$ respectively. 
	The matrix $Y(t)$ is generically full-rank for every $t\in\pac{-1.1,1.1}$ and is smooth, thus owing to \cite[Proposition 2.3]{lucaDieci} there exist unique smooth curves $t\to Q(t)\in\stiefel{n}{k}$ and $t\to R(t) \in \matr{k}{k}$ with positive diagonal entries, such that $Y(t) = Q(t)R(t)$. The positivity of the diagonal entries is explicitly enforced in the experiment. We focus on interpolating the curve $Q(t)$ on $\stiefel{n}{k}$. At each sample location $t_i$ we store $p_i = Q(t_i)$ and $v_i = \dot Q(t_i)$ obtained with~\eqref{eq:finiteDifferenceOnManif}.
	
	In Figure~\ref{fig:stiefelFullError}, we plot the pointwise and derivative error as a function of the curve parameter $t$ when interpolating~\eqref{eq:qfactorInterpolationInstance} with different schemes. While all schemes interpolate correctly the data points (left panel), only the RH scheme manages to match the derivative at sample points (right panel). Figure~\ref{fig:stiefelErrorConvergence} illustrates the result of Theorem~\ref{theo:convergenceOrder}. Plotting the maximum pointwise interpolation error against the sampling step sizes $h$ reveals the expected $O(h^4)$ trend for the RH scheme. Interestingly, as in the Euclidean case, the derivative error converges one order slower than the pointwise error. For these experiments we used the P-factor retraction, but analogous result are found with the Q-factor retraction. The difference between the two retractions is also negligible in terms of evaluation time as it can be seen from Table~\ref{tab:CPUTimes}. These results also show that with the offline/online procedure proposed in Algorithms~\ref{alg:offlinePhase} and~\ref{alg:onlinePhase}, the evaluation cost of the RH scheme is comparable with the one of other schemes. Note that, for a fair comparison, the other schemes have also been implemented in an offline/online fashion to minimize evaluation cost.  
	\begin{figure}
		\centering
		\begin{subfigure}{0.495\linewidth}
			\includegraphics[width=\linewidth, trim=1.55cm 0 0 0, clip]{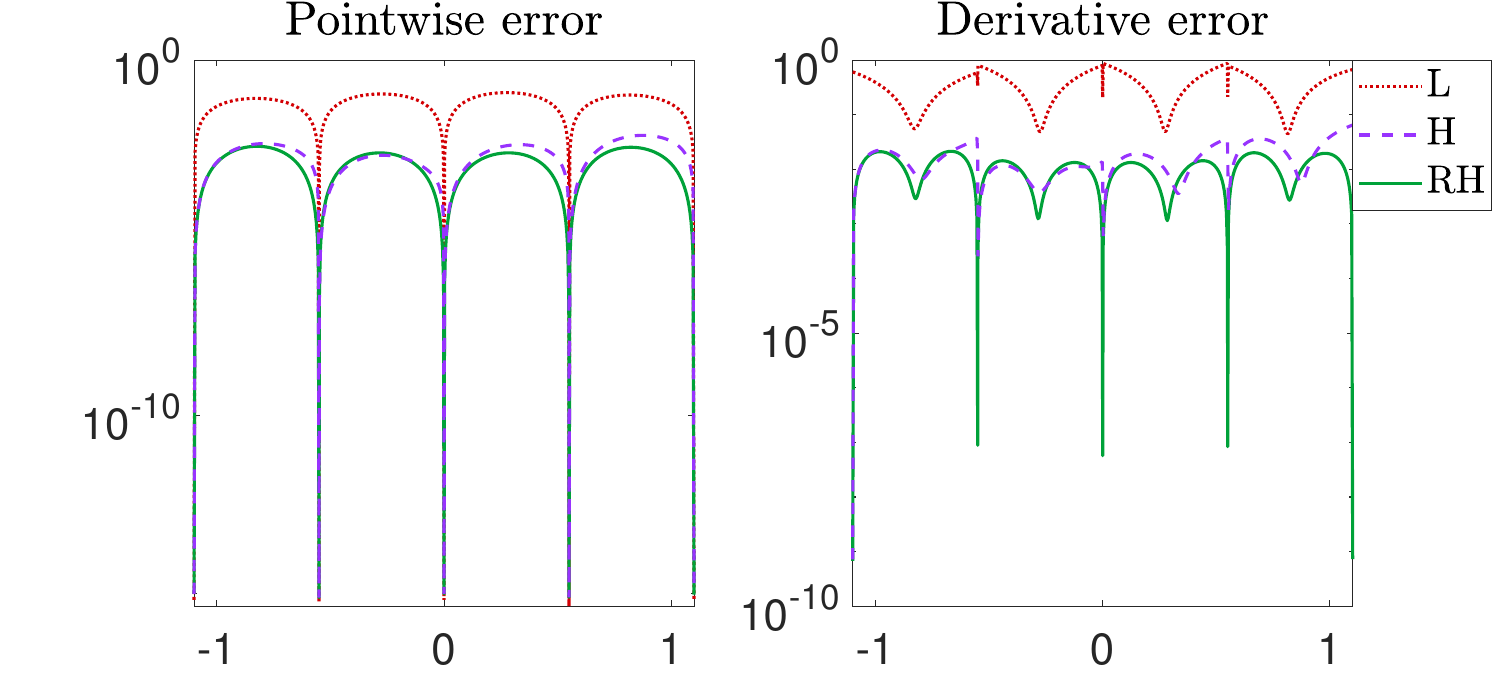}
			\subcaption{Q-factor interpolation.}
			\label{fig:stiefelFullError}
		\end{subfigure}
		\begin{subfigure}{0.495\linewidth}
			\includegraphics[width=\linewidth, trim=1.55cm 0 0 0, clip]{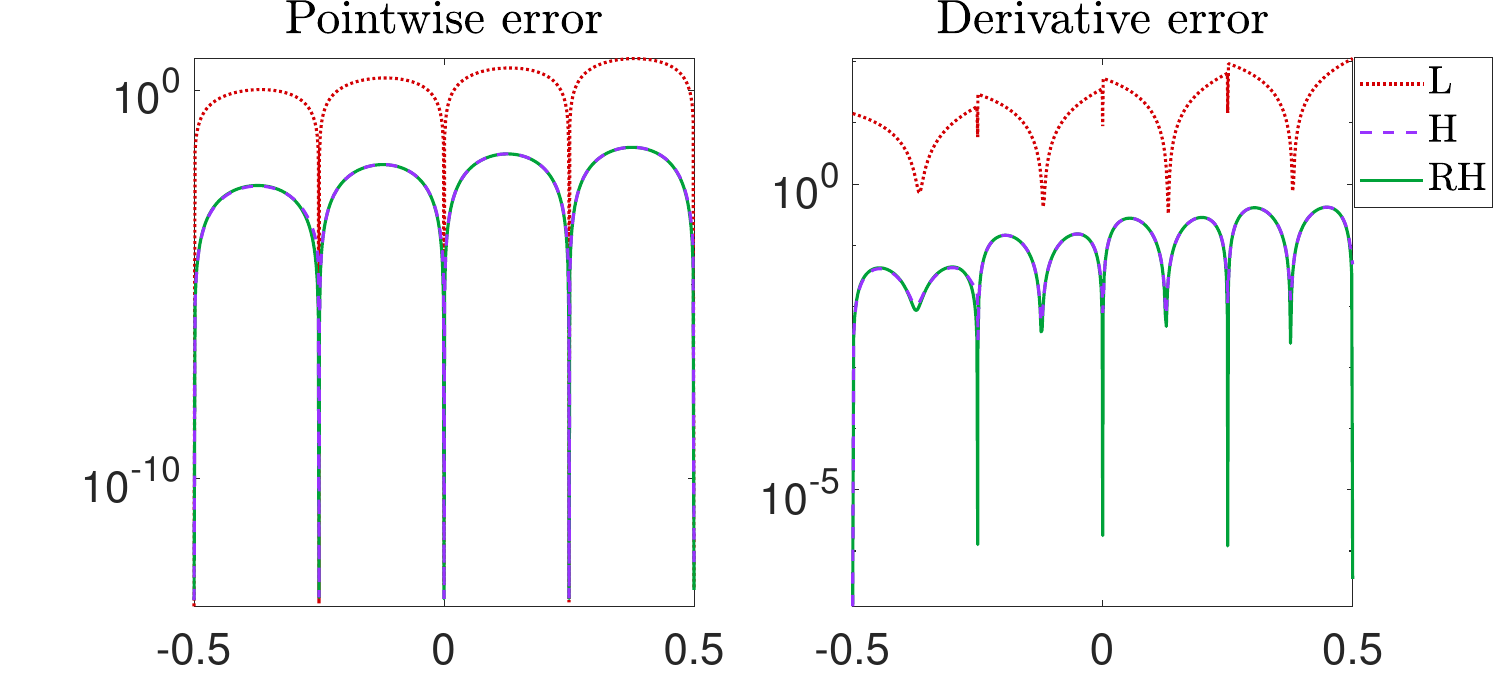}
			\subcaption{SVD interpolation.}
			\label{fig:lowRankFullError}
		\end{subfigure}
		\caption{Interpolation error vs $t$ for different retraction-based schemes: linear (L) and Hermite (H), as defined in Section~\ref{sss:otherSchemes}, and the RH scheme. }
	\end{figure}
	\begin{figure}
		\centering
		\begin{subfigure}{0.495\linewidth}
			\includegraphics[height=5.98cm,trim=0.425cm 0.39cm 5.63cm 0.5cm, clip]{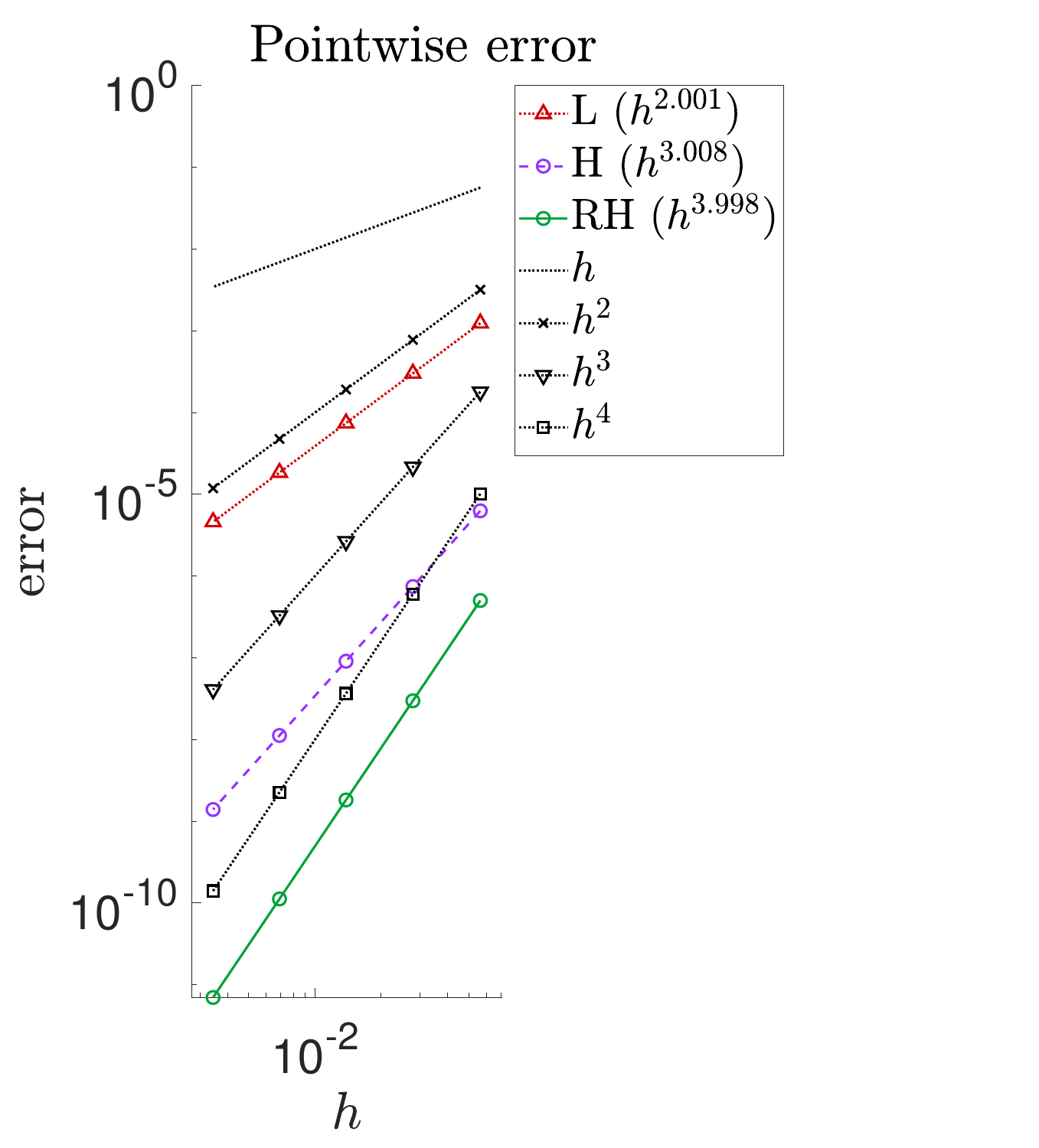}
			\includegraphics[height=5.98cm,trim=1.425cm 0.39cm 5.63cm 0.5cm, clip]{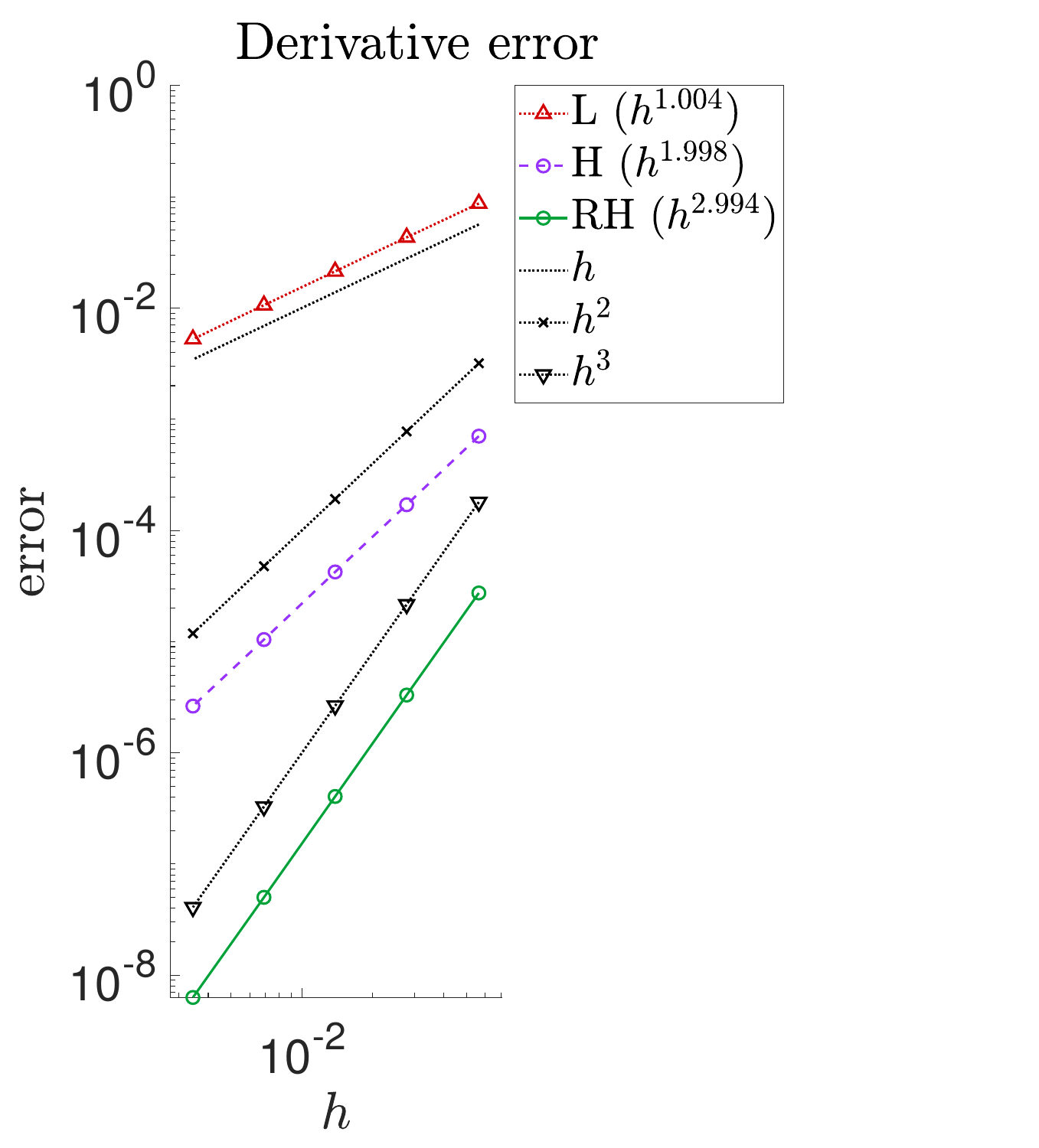}
			\subcaption{Q-factor interpolation.}
			\label{fig:stiefelErrorConvergence}	
		\end{subfigure}
		\begin{subfigure}{0.495\linewidth}
			\includegraphics[height=5.98cm,trim=1.425cm 0.39cm 5.63cm 0.5cm, clip]{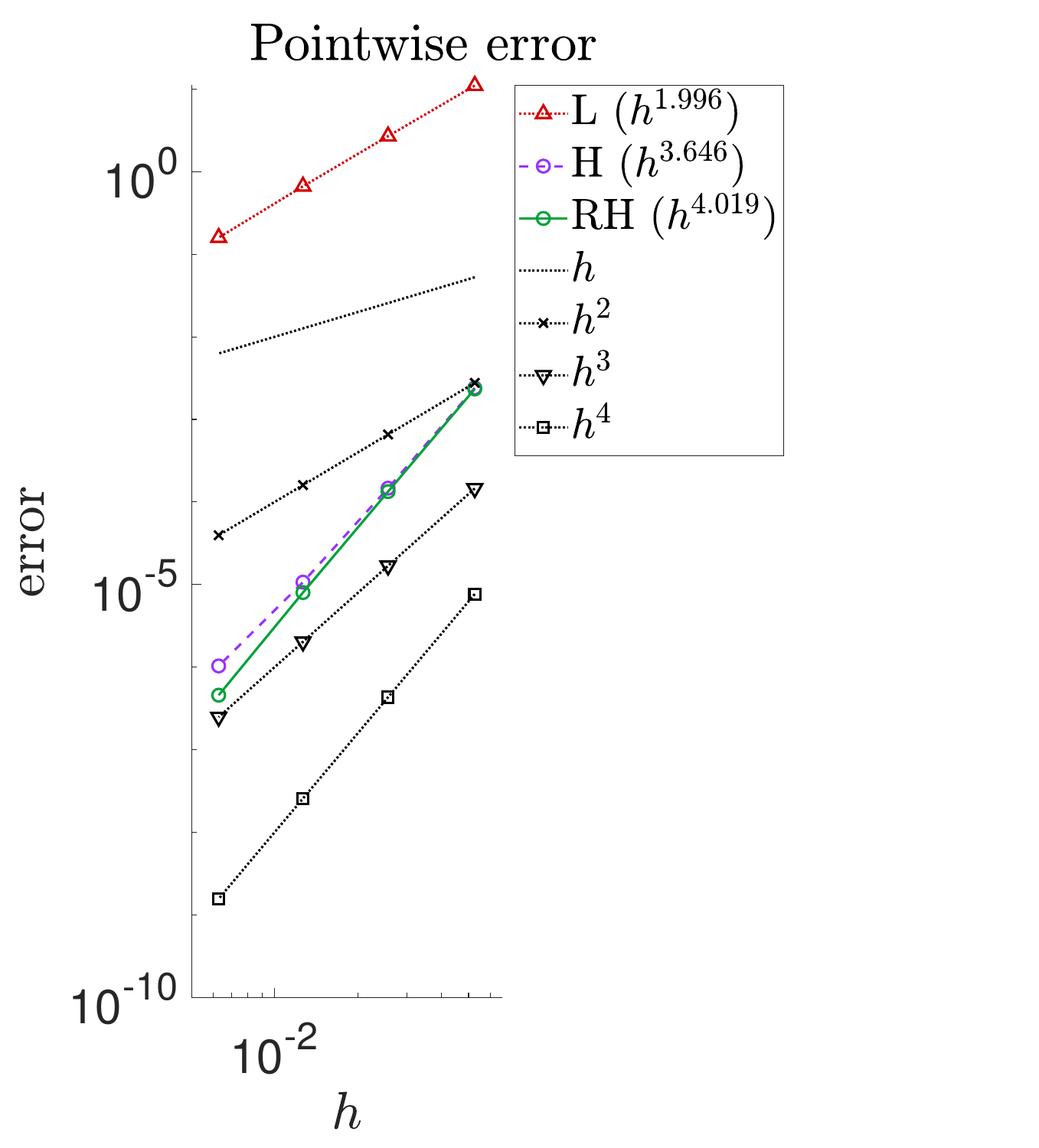}
			\includegraphics[height=5.98cm,trim=1.425cm 0.39cm 5.63cm 0.5cm, clip]{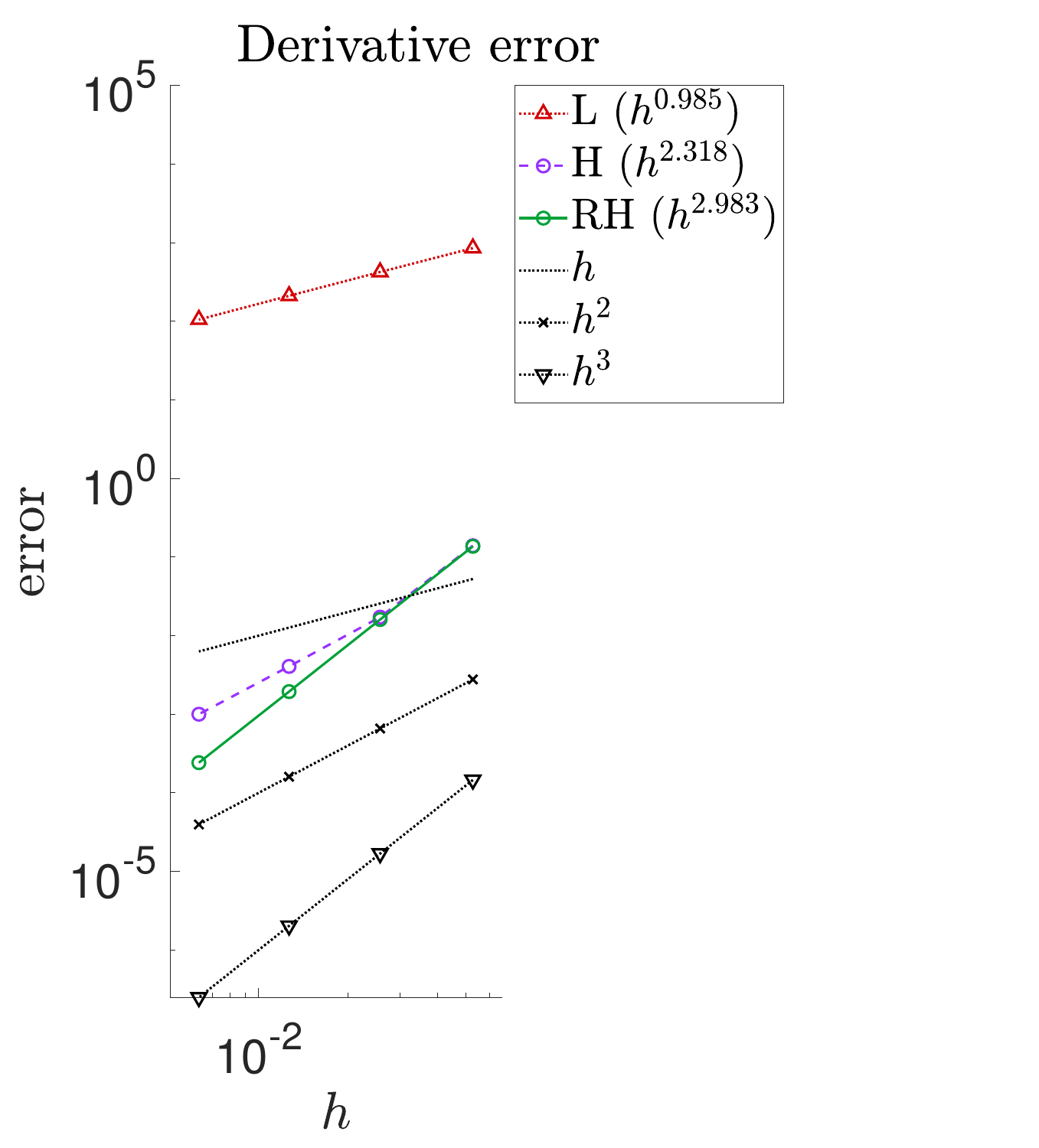}
			\subcaption{SVD interpolation.}
			\label{fig:lowRankErrorConvergence}	
		\end{subfigure}
		\caption{Interpolation error as a function of the sampling step size $h$ for different retraction-based schemes: linear (L) and Hermite (H), as defined in Section~\ref{sss:otherSchemes}, and the RH scheme. }
	\end{figure}
	\begin{table}[]
		\centering
		\caption{Average time per evaluation for the Q-factor and SVD interpolation experiments of Figure~\ref{fig:stiefelFullError} and Figure~\ref{fig:lowRankFullError}. In the last two lines we distinguish the naive and the optimized offline/online implementations of the RH scheme (see Algorithms~\ref{alg:offlinePhase} and~\ref{alg:onlinePhase}). }
		\label{tab:CPUTimes}
		\begin{tabular}{c|cc|c|}
			\cline{2-4}
			& \multicolumn{2}{c|}{$Q$-factor interpolation} & SVD interpolation \\ \hline
			\multicolumn{1}{|c|}{Retraction used} & \multicolumn{1}{c|}{$Q$-factor} & $P$-factor & Orthographic \\ \hline 
			\multicolumn{1}{|c|}{Linear (L)} & \multicolumn{1}{c|}{$1.524\cdot10^{-4}$} & $2.049\cdot 10^{-4}$ & $1.710\cdot 10^{-3}$ \\ \hline
			\multicolumn{1}{|c|}{Hermite (H)} & \multicolumn{1}{c|}{$1.540\cdot 10^{-3}$} & $1.596\cdot 10^{-3}$ & $1.147\cdot 10^{-2}$ \\ \hline
			\multicolumn{1}{|c|}{Retraction Hermite (RH)} & \multicolumn{1}{c|}{$4.335\cdot 10^{-3}$} & $3.531\cdot 10^{-3}$ & $1.992\cdot 10^{-2}$ \\ \hline
			\multicolumn{1}{|c|}{Retraction Hermite (RH, optimized)} & \multicolumn{1}{c|}{$2.226\cdot 10^{-3}$} & $2.188\cdot 10^{-3}$ & $1.558\cdot 10^{-2}$ \\ \hline
		\end{tabular}
	\end{table}
	
	\subsubsection{SVD interpolation}
	
	We interpolate the singular value decomposition of a matrix curve of constant low rank. We consider $m = 10^4$, $n = 300$ and rank $r = 10$. We take 
	\begin{align*}
		Y(t) = Y_0 + t Y_1 + t^2 Y_2 + t^3 Y_3, \quad Y_i\in \matr{m}{r} \\
		Z(t) = Z_0 + t Z_1 + t^2 Z_2, \quad Z_i \in \revision{\matr{n}{r}}, 
	\end{align*}
	where the entries of $Y_0, Z_0$ and $Y_1,Y_2,Y_3,Z_1,Z_2$ are drawn \revision{from} uniform random distributions on $\pac{0,1}$ and $\pac{0,0.5}$, respectively. Since the factors are generically \revision{of} full rank, the curve 
	\begin{equation}\label{eq:svdInterpInstance}
		W(t) = Y(t) Z(t)^\top,\quad t\in\pac{-0.5,0.5}
	\end{equation}
	is of rank $r$ for every $t$. This example has also been taken from \cite[\S 5.3]{zimmermann} together with the suggestion of Remark 6 \revision{from~\cite{zimmermann}}  to ensure the smoothness of the computed SVD decomposition path $U(t)\Sigma(t) V(t)^\top = W(t)$. Note that this may cause negative values in the diagonal term.  
	
	The experimental results for the SVD path of~\eqref{eq:svdInterpInstance} are reported in Figures~\ref{fig:lowRankFullError} and~\ref{fig:lowRankErrorConvergence}. The comments are analogous to the one we made for the Q-factor interpolation experiments: only the RH scheme manages to match the prescribed derivatives at the nodes thereby producing the expected $O(h^4)$ error convergence trend. For these experiments, the naive retraction-based Hermite interpolation scheme H produces a good approximation in terms of error, practically as good as the RH scheme. However, the curve H is not globally $C^1$.
	
	The computation time per evaluation for the different schemes is reported in Table~\ref{tab:CPUTimes}. Interestingly, while for the Stiefel manifold, a naive implementation of the RH scheme is approximately two times slower than the offline/online approach, for the low-rank manifold, the non-optimized code is only 30\% more expensive. We attribute this to the fact the inverse orthographic retraction is relatively cheap compared to the retraction and so the few inverse retractions spared by the offline/online implementation do not pay off as much. 
	\subsubsection{The need for bounded derivatives}\label{ss:needForBoundedDerivatives}
	The fourth-order convergence achieved by the RH interpolation scheme was proved in Theorem~\ref{theo:convergenceOrder} under the assumption that all derivatives up to order 4 of the interpolation curve remain bounded as  $h\to 0$, see~\eqref{eq:boundednessOfKthDerivOfHh}. As we now illustrate, this assumption can in fact not be removed. Recall that the RH interpolation scheme was built to satisfy Proposition~\ref{prop:retractionDeCast} by making the choice~\eqref{eq:bestChoiceOfr1r01r12r012}. However, this choice was not unique and was motivated by the need to alleviate evaluation cost. It turns out that choice~\eqref{eq:bestChoiceOfr1r01r12r012} is also important because it satisfies the bounded derivatives assumption~\eqref{eq:boundednessOfKthDerivOfHh}. In Figure~\ref{fig:maxDerivUnbounded} we plot the maximum norm of the second, third and fourth order derivatives as a function of the sampling size $h$ for the RH scheme and for an alternative scheme denoted RH$^*$ where choice~\eqref{eq:bestChoiceOfr1r01r12r012} is modified with $r_1 = 0$ instead of $r_1 = 1/2$. It appears that unlike the RH scheme, the RH$^*$ scheme features a fourth derivative diverging as $O(h^{-1})$. Despite producing a continuously differentiable curve which interpolates the derivatives, the alternative scheme loses one order of accuracy as can be seen from Figure~\ref{fig:errorConvergenceBad}. These experiments were conducted on the Q-factor interpolation of~\eqref{eq:qfactorInterpolationInstance} on the Stiefel manifold but analogous results were found on the SVD interpolation instance. 
	
	\begin{figure}
		\centering
		\begin{subfigure}{0.49\linewidth}
			\centering
			\includegraphics[height=3.9cm,trim=0 0 1.05cm 0, clip]{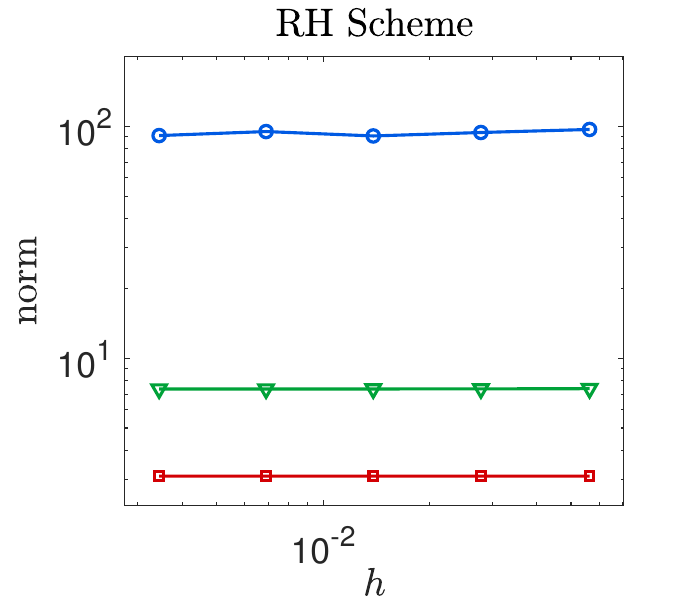} \includegraphics[height=3.9cm,trim=0.45cm 0 0.9cm 0, clip]{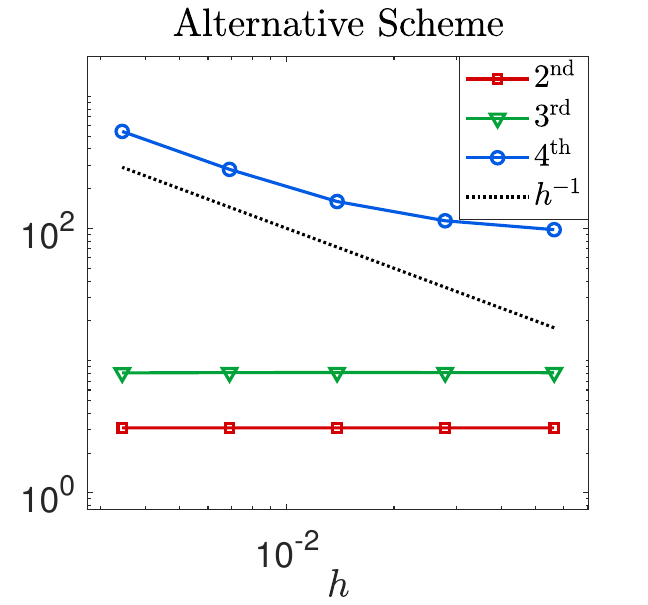}
			\caption{Maximum derivative vs $h$.}
			\label{fig:maxDerivUnbounded}
		\end{subfigure}
		\begin{subfigure}{0.49\linewidth}
			\centering
			\includegraphics[height=3.9cm,trim=0 0 5.5cm 0, clip]{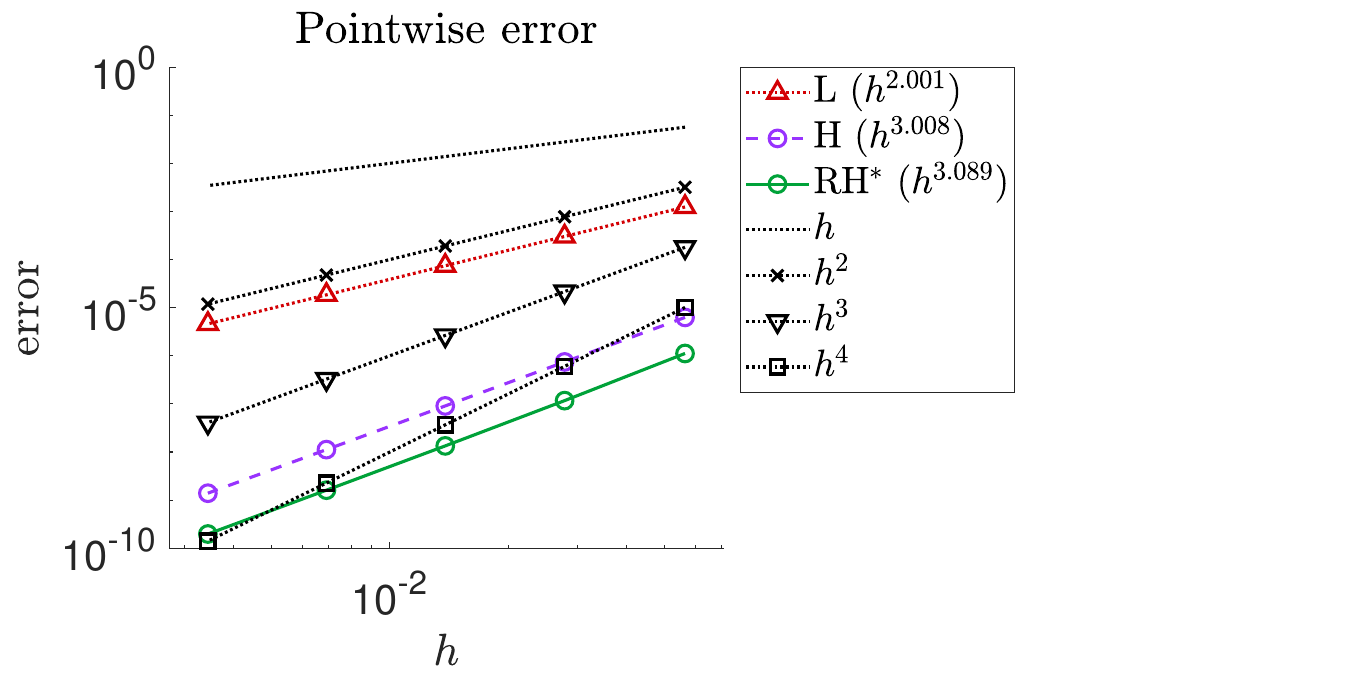}
			\caption{Alternative scheme error convergence.}
			\label{fig:errorConvergenceBad}
		\end{subfigure}
		\caption{Experiments with an alternative interpolation scheme still satisfying Proposition~\ref{prop:retractionDeCast} but failing to satisfy assumption~\eqref{eq:boundednessOfKthDerivOfHh}.}
	\end{figure}
	
	\revision{Among the admissible choice determined by Proposition~\ref{prop:retractionDeCast}}, the choice~\eqref{eq:bestChoiceOfr1r01r12r012} is not the only one satisfying assumption~\eqref{eq:boundednessOfKthDerivOfHh}. For instance, for any choice of $r_1\in\pac{0,1}$ if we take $r_{01}(s,t) = (1-r_1)t$, $r_{12}(s,t) = (1-r_1) + r_1 t$ and $r_{012}(u,s,t) = t$ we still achieve $O(h^4)$ error convergence. However, the evaluation cost of such schemes is higher and among those that we could find, choice~\eqref{eq:bestChoiceOfr1r01r12r012} was the cheapest \revision{as its evaluation requires the least number of retraction and inverse retractions}. The relationship between the choice of these functions and the fourth derivative of the scheme is intricate and we could not establish an a priori criterion to discriminate between schemes satisfying~\eqref{eq:boundednessOfKthDerivOfHh} and those violating it.
	
	\subsection{Applications}
	In this section, we illustrate two possible applications of the RH interpolation scheme. We focus on applications involving the fixed-rank manifold as we believe they are the most relevant.
	
	\subsubsection{Riemannian continuation}
	
	As a first application, we propose an extension of the numerical continuation \revision{method for Riemannian optimization from~\cite{myFirstPaper}. This application considers the minimization of a scalar target function $f(x,\lambda)$ for $x \in \mathcal M$ smoothly parametrized by $\lambda \in [0,1]$. For example, this is useful when the 
	minimum is simple to obtain for $\lambda = 0$ and one successively finds the minimum for increasing values of $\lambda$ until one reaches $\lambda = 1$, corresponding to the minimization problem of interest. Suppose one has computed solutions $\ldots, x_{k-1}, x_k$ at previous parameter instances $\ldots, \lambda_{k-1}, \lambda_k$ the computation of $x_{k+1}$ at $\lambda_{k+1} = \lambda_k + h_k$ proceeds in two steps. The prediction step aims to attain a relatively cheap, yet accurate approximation $y_{k+1}$. This is followed by a standard Riemannian optimization method, warm started with the excellent initial guess $y_{k+1}$.
	}
	 The prediction step of the continuation algorithm can be characterized by the so-called prediction order~\cite[Definition 3.2]{myFirstPaper}. It captures how accurately the prediction approximates the solution to the next problem. Theorem~\ref{theo:convergenceOrder} indicates that using the RH interpolation scheme in the prediction step produces a prediction order $p_{RH} = 4$ compared to $p_C = 1$ and $p_T = 2$ for the classical and tangential prediction schemes \revision{discussed in~\cite{myFirstPaper}}. 
	%In practice, in the notation of~\cite[Algorithm 3.2]{myFirstPaper} with
	\revision{Choosing} constant step size $h_k = h$ for every $k\geq 0$, \revision{and introducing the notation $\tau\mapsto H(\tau;\paa{p_0,v_0,\tau_0}, \paa{p_1,v_1,\tau_1})$ to indicate the RH interpolant passing through $p_i$ with velocity $v_i$ at $\tau=\tau_i$, for $i = 0,1$,}  we define the RH prediction step as 
	\begin{equation*}
		y_{k+1} = \begin{cases}
			R_{x_{k}}\pa{h t_k}\quad &k = 0,\\
			H(2h;\paa{x_{k-1},t_{k-1},0} \paa{x_k,t_k,h}),\quad &k\geq 1,
		\end{cases}
	\end{equation*}
	\revision{For $k = 0$, this is standard tangential prediction and, for $k\ge 1$, we employ a RH interpolant and to provide an initial guess to the next problem. Note that the RH interpolant itself is constructed on the interval $\pac{0,h}$ as specified by Corollary~\ref{cor:hermiteRetrFullSolution}. Yet, provided the retractions used in the procedure remain well-defined, we can  evaluate it outside of this interval, in particular at $\tau = 2h$. This requires evaluating the elementary endpoint curves of the de Casteljau algorithm appearing in the lines 4 to 9 of Algorithm~\ref{alg:onlinePhase} outside the interval $\pac{0,1}$.}

	The RH prediction-correction continuation algorithm with fixed step size is applied to the same low-rank matrix completion problem considered in~\cite[\S 5]{myFirstPaper}. We fix the number of steps to $N_{\text{steps}} = 5$, use the Riemannian Trust Regions (RTR) algorithm as a corrector and vary the prediction scheme. We report in \revision{T}able~\ref{tab:rhContinnuation} a comparison of the computational effort required to solve the problem with each scheme. There are two factors that determine the performance of the algorithm. First, the more ill-conditioned the final optimization problem is, the more the last RTR correction will encounter stagnation. Second, the more the underlying solution curve to the family of optimization problem is smooth, the more tangential and RH prediction will pay off. In fact, high prediction order will be achieved only when the underlying solution curve to the family of optimization problems is sufficiently smooth. Partly due to the choice of the homotopy, the solution curve often exhibited discontinuities thereby undermining the efficiency of tangential and RH prediction. Because of this, these two schemes were on average slower than classical prediction. However, when the underlying solution curve happens to be smooth, e.g. for the instance of Figure~\ref{fig:rhContinuation}, the RH prediction benefits the increased prediction accuracy. As can be seen from the computed medians for computation time and total RTR iteration count, the RH can significantly reduce the computational effort compared to classical and tangential prediction. 
	\begin{figure}
		\includegraphics[width=0.495\linewidth]{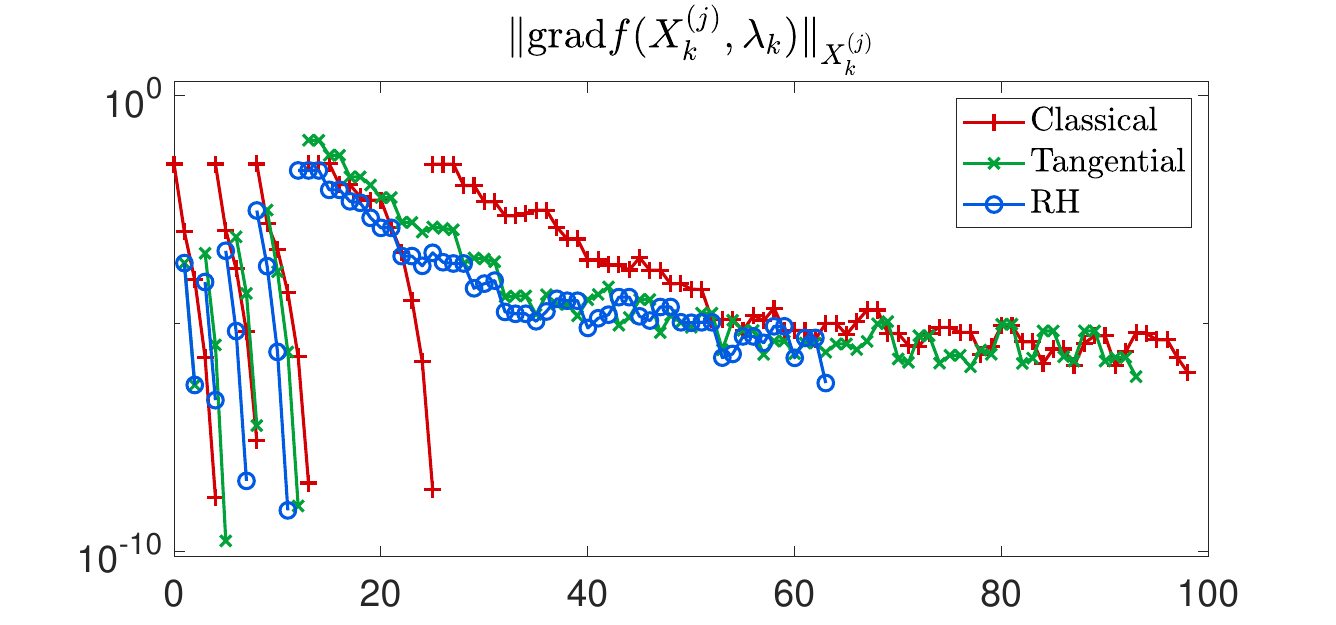}
		\includegraphics[width=0.495\linewidth]{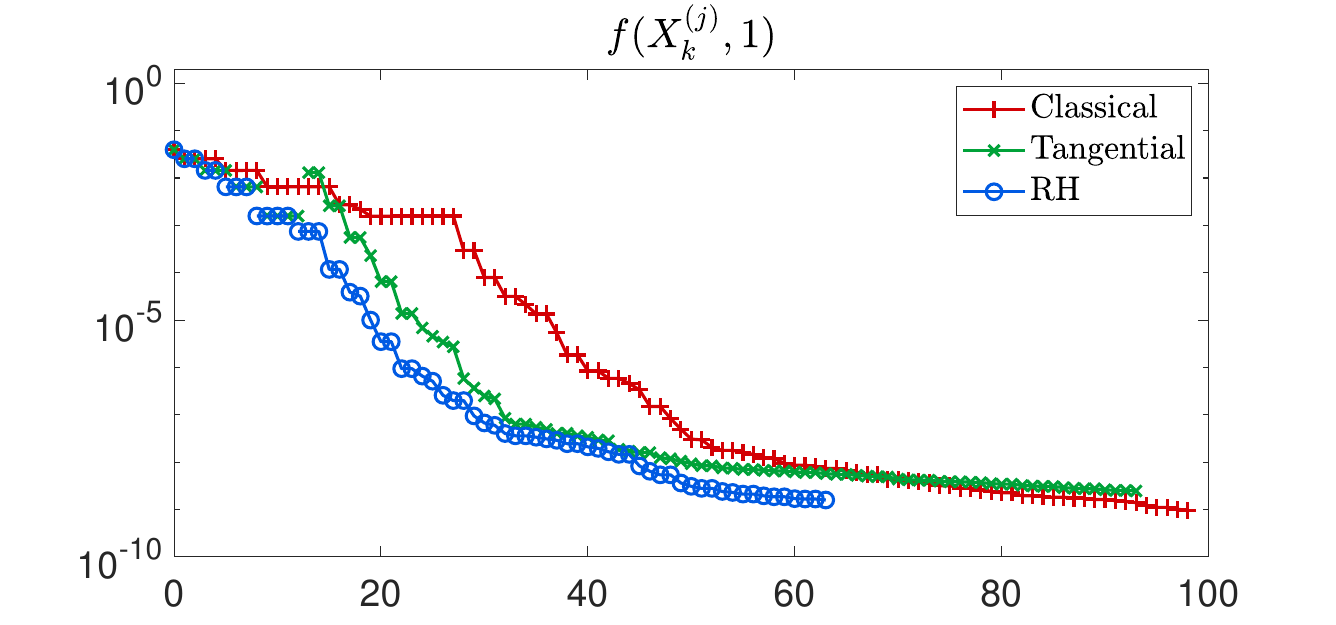}
		\caption{RH continuation algorithm on a favorable instance of matrix completion. Left: Riemannian gradient of the sequence of problems versus total iteration count. Right: objective function value of the final optimization problem versus total iteration count.}
		%Discontinuous portions of the curves with the same color indicate different problems in the sequence of optimization problems. 
		\label{fig:rhContinuation}
	\end{figure}
	\begin{table}
		\caption{Performance statistics of the Riemannian continuation algorithm for three different prediction schemes over 100 randomly generated instances of the matrix completion problem. }
		\label{tab:rhContinnuation}
		\centering
		\revision{\begin{tabular}{c|ccc|ccc|}
			\cline{2-7}
			& \multicolumn{3}{c|}{Time (s)} & \multicolumn{3}{c|}{Correction steps (RTR)} \\ \cline{2-7} 
			& \multicolumn{1}{c|}{Mean} & \multicolumn{1}{c|}{Median} & Std. dev. & \multicolumn{1}{c|}{Mean} & \multicolumn{1}{c|}{Median} & Std. dev. \\ \hline
			\multicolumn{1}{|c|}{Classical} & \multicolumn{1}{c|}{3.512} & \multicolumn{1}{c|}{3.435} & 4.035 & \multicolumn{1}{c|}{59.56} & \multicolumn{1}{c|}{60} & 20.10  \\ \hline
			\multicolumn{1}{|c|}{Tangential} & \multicolumn{1}{c|}{7.919} & \multicolumn{1}{c|}{4.706} & 6.735 & \multicolumn{1}{c|}{69.95} & \multicolumn{1}{c|}{53} & 66.16  \\ \hline
			\multicolumn{1}{|c|}{RH} & \multicolumn{1}{c|}{3.770} & \multicolumn{1}{c|}{2.588} & 9.159 & \multicolumn{1}{c|}{48.72} & \multicolumn{1}{c|}{45} & 79.69 \\ \hline
		\end{tabular}}
	\end{table}
	
	\subsubsection{Dynamical-low-rank approximation interpolation}
	
	Dynamical low-rank approximation~\cite{dlraRefPaper} (DLRA) is a \revision{numerical integration technique for solving a matrix differential equation of the form
	\begin{equation} \label{eq:mldra}
		\dot Y = F(Y,t), \quad Y(0) = Y_0, \quad Y(t) \in\matr{m}{n}, \quad t\in\pac{0,t_{\text{end}}}.
	\end{equation}
	Such equations arise, for example, from the structured discretization of two-dimensional partial differential equations. DLRA assumes that $Y(t)$ admits, for every $t$, an accurate rank-$k$ approximation with $k \ll \min\{m,n\}$. In particular, this allows one to replace $Y_0$ by its best rank-$k$ approximation $\tilde Y_0$ (using the singular value decomposition). Letting $\mathcal M_k$ denote the manifold of rank-$k$ matrices, one replaces the dynamics~\eqref{eq:mldra} by
	$\dot{ \tilde Y} = \Pi_{\tilde Y}(F(\tilde Y,t))$, where $\Pi_{\tilde Y}$ denotes the orthogonal projection onto the tangent space $T_{\tilde Y}\mcal_k$, to 
	ensure that $\tilde Y(t)$ remains on $\mathcal M_k$. This can significantly reduce the computational cost and memory requirements for numerical integration. In practice, some care is required in order to integrate the reduced differential equation on the manifold accurately and efficiently; see, e.g.,~\cite{unconventionalIntegrator}. Such an integrator returns} a sequence $\paa{\tilde Y_i}_{i = 0}^{N_t}\subset \mcal_k$, such that $\tilde Y_i\simeq Y(t_i)$, the solution at time $t_i$ of the matrix differential equation.  During the numerical integration, for each $\tilde Y_i$ we can store the best approximation of the vector $F(Y_i,t_i)$ on the tangent space $T_{\tilde Y_i}\mcal_k$, obtained as $\tilde V_i := \Pi_{\tilde Y_i}\pa{F(\tilde Y_i,t_i)}$. Then, the collection of triplets $\paa{\pa{\tilde Y_i,\tilde V_i,t_i}}_{i=0}^{N_t}$ can be fed to our RH interpolation scheme to obtain a continuously differentiable curve on $\mcal_k$ that approximates the best rank $k$ solution for every time $t$. Given the high accuracy of the interpolation scheme, one can expect that it is sufficient to interpolate a small fraction of the time samples to obtain a satisfactory approximation of the full solution curve. 
	
	\begin{figure}
		\centering
		\includegraphics[width=0.8\linewidth]{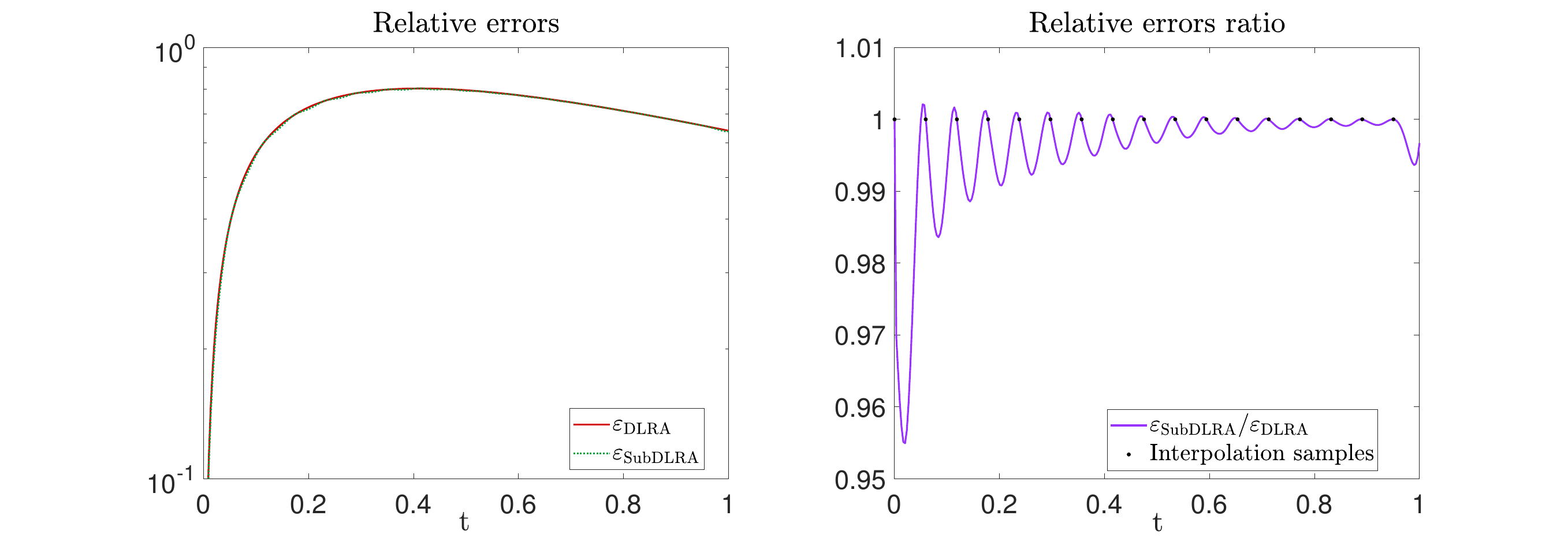}
		\caption{Relative error with reference solution of the interpolation of DLRA samples. The error $\varepsilon_{\text{DLRA}}$ corresponds to interpolating all samples whereas $\varepsilon_{\text{SubDLRA}}$ only 1 out of 20 samples.}
		\label{fig:subSamplingDLR}
	\end{figure}

	We tested this hypothesis on the test case of \cite[\S 7.1]{gianlucaPaperDLR} for which the DLRA integrator is the so-called unconventional time integrator \cite{unconventionalIntegrator}, which features improved stability properties compared to the matrix-projector splitting integrator. For conciseness, we refer the reader to these two references for details on the radiation transport equation at hand and on the numerical integrator. The parameters for the experiments are the same as in \cite[\S 7.1]{gianlucaPaperDLR}: the integration is done on $\mcal_k\subset\matr{m}{n}$ with $m = 800$, $n = 100$ and $k = 15$. The only difference is that we turn off the rank adaptivity option of the integrator since we want the initial condition and all iterates to remain on $\mcal_k$. In order to obtain a rank 15 initial condition, we run the rank adaptive version starting from the rank 1 initial condition used in \cite[\S 7.1]{gianlucaPaperDLR} and store the first sample of rank 15. From this initial condition, we run the unconventional DLRA integrator with a step size \revision{$\Delta t = 3 \times 10^{-3}$ for $N_t$ steps up to $T = 1$}. This yields $\paa{\tilde Y_i}_{i = 0}^{N_t}\subset \mcal_k$. We then obtain a reference solution by performing a forward Euler integration of the ODE in $\matr{m}{n}$ with a step size 20 times smaller than the unconventional integrator. These samples are then projected to $\mcal_k$ with the $k$-truncated SVD, and we denote them $\paa{Y_j}_{j  = 0}^{20N_t}$. Finally, we assemble three RH interpolants:  
	\begin{align*}
		\tilde Y(t) &= H\pa{t;\paa{\pa{\tilde Y_i,\tilde V_i, t_i}}_{i = 0}^{N_t}},\\
		\hat Y(t) &= H\pa{t;\paa{\pa{\tilde Y_{20l},\tilde V_{20l}, t_{20l}}}_{l = 0}^{\floor{N_t / 20}}},\\
		Y(t) &= H\pa{t;\paa{ \pa{Y_j, V_j, t_j}}_{j = 0}^{20 N_t}}
	\end{align*}
	where tangent vectors $\tilde V_i$ and $V_i$ are obtained as explained previously. The curve $\hat Y$ interpolates one every 20 samples of the DLRA solution, so roughly 5\% of the integrator's output. Yet, as can be seen from Figure~\ref{fig:subSamplingDLR}, the relative errors
	\begin{equation*}
		\varepsilon_{\text{DLRA}}(t) = \frac{\norm{\tilde Y(t) - Y(t)}_F}{\norm{Y(t)}_F},\quad \varepsilon_{\text{SubDLRA}}(t) = \frac{\norm{\hat Y(t) - Y(t)}_F}{\norm{Y(t)}_F},
	\end{equation*}
	are almost identical. Surprisingly, the sub-sampled interpolation curve $\hat Y$ can be more precise, though from a negligible amount. The real advantage comes in terms of storage requirements: the information needed to evaluate the sub-sampled interpolation curve $\hat Y$ occupies 20 times less memory than the information for the curve $\tilde Y$ and approximately 4 times less storage than the full collection of samples $\paa{\tilde Y_i}_{i = 0}^{N_t}$. This application of RH interpolation can be thought as a compression post-processing that enhances the portability of DLRA solutions.

	\section{Conclusion}
	
	We have proposed a manifold interpolation technique to address Hermite interpolation of manifold curves. The method is general enough to be applicable to every manifold for which a retraction/inverse retraction pair is available, thereby avoiding the need for Riemannian exponential and logarithmic maps used by other interpolation schemes that take into account derivative information. 
	
	The novel notion of retraction-convex sets ensures the well-posedness of the method, provided that the interpolation data is sufficiently close. Theorem~\ref{theo:convergenceOrder} generalizes to our scheme the classical interpolation error convergence result for polynomial Hermite interpolation in Euclidean spaces. The predicted $O(h^4)$ convergence trend has also been experimentally observed for academic interpolation problems on the Stiefel manifold and the fixed-rank matrix manifold. The high order accuracy of the method allowed us to propose an improvement to the prediction-correction Riemannian continuation algorithm of~\cite{myFirstPaper} and to suggest a strategy to compress the output of dynamical low-rank matrix integrators. 
	
	Just like curve interpolation is a basic tool in many context of numerical analysis, we believe the RH interpolation scheme could serve as a building block for many other numerical methods on manifolds, for example the numerical integration of manifold ODEs.

	\paragraph{Acknowledgements} The authors would like to thank Ralf Zimmerman for helpful discussions on the topic and Gianluca Ceruti for suggesting the DLRA application as well as providing the code for the unconventional integrator from~\cite{gianlucaPaperDLR}. The authors thank the reviewer for very helpful remarks.
	
	{\small\noindent\textbf{Conflicts of interest}$\:\:$  The authors have no conflicts of interest to declare that are relevant to the content of this article}.
	
    \bibliography{bibliographyWithBibTex}

    \appendix
    
    \section{Lipschitz continuity of the retraction}\label{app:lipschitzContinuityOfRetraction}
    
    In this section, we establish the Lipschitz continuity properties of retractions stated in Section~\ref{ss:lipschitzContinuityOfRetraction} and used in Appendix~\ref{app:lipschContIndepOfH} below.
    
    Recall that Proposition~\ref{prop:retractionDiffeom} yields a suitable domain $\tcal\subset T\mcal$ of invertibility of the retraction. We set ${\tcal_x := \piOnT{x}\pa{\tcal}}$
    and $\ical_x := R_x(\tcal_x)$ for $x\in\mcal$. For any $\varepsilon \in \pa{0,1}$ we define 
    \begin{equation}
    	\tcal^\varepsilon := \paa{(x,v)\in T\mcal :  \|v\|_x<\varepsilon\Delta(x)},
    \end{equation} 
    an open set that satisfies $\overline{\tcal^\varepsilon} \subset \tcal$. Similarly, we set $\tcal_x^\varepsilon := \pi_{T_x\mcal}(\tcal^\varepsilon)$ and $\ical_x^\varepsilon = R_x(\tcal_x^\varepsilon)$.
    
    We consider differentials of the retraction and inverse retraction with respect to their arguments. Given $x\in\mcal$, $u\in\tcal_x$, $v\in T_x\mcal$, $y\in\ical_x$ and $w\in T_y\mcal$ we denote
    \begin{align}
    	\D_1 R_x(u)[v] &= \ddt{}{t}R_{\sigma_{x,v}(t)}(u)\restr{t = 0}\in T_{R_x(u)}\mcal,\\
    	\D_2 R_x(u)[v] &= \ddt{}{t}R_x(u + tv)\restr{t = 0}\in T_{R_x(u)}\mcal,\\
    	\D_1 R_x^{-1}(y)[v] &=  \ddt{}{t}R_{\sigma_{x,v}(t)}^{-1}(y)\restr{t = 0}\in T_x\mcal,\\
    	\D_2 R_x^{-1}(y)[w] &=  \ddt{}{t}R_{x}^{-1}(\sigma_{y,w}(t))\restr{t = 0}\in T_x\mcal,
    \end{align}
    where $\sigma_{x,v}$ is any continuously differentiable manifold curve defined in a neighborhood of $t = 0$ such that $\sigma_{x,v}(0) = x$, $\dot\sigma_{x,v}(0) = v$.
    
    \begin{proposition}\label{prop:lipschRetr}
    	Consider arbitrary $\varepsilon\in\pa{0,1}$. For any $x\in\mcal$, there exist positive constants $L_1(x,\varepsilon)$, $M_1(x,\varepsilon)$, $L_2(x,\varepsilon)$ and $M_2(x,\varepsilon)$ such that the retraction $R$ satisfies 
    	\begin{enumerate}[label=(\roman*)]
    		\item \label{prop:lipschR1} 
    		$
    		\norm{\D_1 R_x(u)[v]}_{R_x(u)} \leq L_1(x,\varepsilon)\|v\|_x
    		$ for every $u\in\tcal_x^\varepsilon$ and $v\in T_x\mcal$
    		\item \label{prop:lipschR2} $\norm{\D_2 R_x(u)[v]}_{R_x(u)}\leq L_2(x,\varepsilon) \|v\|_x$ for every $u\in\tcal_x^\varepsilon$ and $v\in  T_uT_x\mcal\simeq T_x\mcal$,
    		\item \label{prop:lipschRinv1} $\norm{\D_1 R_x^{-1}(y)[v]}_x \leq M_1(x,\varepsilon)\|v\|_x$ for every $y\in\ical_x^\varepsilon$ and $v\in T_x\mcal$,
    		\item \label{prop:lipschRinv2} $\norm{\D_2 R_x^{-1}(y)[v]}_x\leq M_2(x,\varepsilon) \|v\|_y$ for every $y\in \ical_x^\varepsilon$ and $v\in T_y\mcal$.
    	\end{enumerate}
    \end{proposition}
    
    \begin{proof}
    	The result follows from the smoothness of the retraction and its inverse. The Lipschitz constants are found by taking the maxima of the operator norm for the differentials of the retraction on the compact set $\overline{\tcal_x^\varepsilon}$ and of the inverse retraction on the compact set $\overline{\ical_x^\varepsilon}$.  \qed
    \end{proof}

    For simplicity of exposition, a particular choice $\varepsilon = 1/3$ is considered in Lemma~\ref{lem:lipschRetrAndInvRetr}. Note however that the result (and the upcoming proof) hold for any $\varepsilon\in\pa{0,\frac{1}{2}}$
    \begin{proof}(of Lemma~\ref{lem:lipschRetrAndInvRetr})
    	\begin{enumerate}[label=(\roman*)]
    		\item Consider any $x\in \mcal$, $u,v\in\tcal_x^{1/3}$ and define a manifold curve joining $R_x(u)$ and $ R_x(v)$ as $\delta(\tau) := R_x(u + \tau(v - u))$. It is well-defined by the convexity of $\tcal_x^{1/3}$. By definition, the manifold distance between $R_x(u)$ and $ R_x(v)$ is bounded by the length $L(\delta)$ of the curve $\delta$. Thus we have 
    		\begin{align}
    			d(R_x(u), R_x(v)) \leq L(\delta) & = \int_{0}^{1}\norm{\D_2R_x(u + \tau(v - u))\pac{v - u}}_{\delta(\tau)}\dd\tau\\
    			&\leq L_2(x,1/3) \|v-u\|_x,
    		\end{align}
    		where the final inequality follows from Proposition~\ref{prop:lipschRetr}-\ref{prop:lipschR2}. Hence $L_R(x) = L_2(x,1/3)$.
    		\item Consider $y,z\in \ical_x^{1/3}$. Since $\mcal$ is connected, by definition of the distance function we know there exists a sequence of piecewise smooth manifold curves $\gamma_k : \pac{0,1}\to\mcal$, $k\in\mathbb{N}$, such that
    		\begin{equation}
    			\gamma_k(0) = y,\quad \gamma_k(1) = z,\quad \forall k\in\mathbb{N}
    		\end{equation}
    		and
    		\begin{equation}
    			\lim_{k\to\infty} L(\gamma_k) = \infOn{k\in\mathbb{N}}\paa{L(\gamma_k)} = d(y,z).
    		\end{equation}
    		If the image of the curve $\gamma_k$, is fully contained in $\ical_x^{2/3}$, we can define the tangent space curve $q(t) := R^{-1}_x(\gamma_k(t))$ for every $t\in\pac{0,1}$ and deduce that
    		\begin{equation*}
    			\begin{aligned}	
    				\norm{R_x^{-1}(y) - R_x^{-1}(z)}_x &\leq \int_0^1 \norm{D_2R^{-1}_x(\gamma_k(\tau))\pac{\dot\gamma_k(\tau)}}_x\dd \tau\\ &\leq M_2(x,2/3)\int_{0}^1\norm{\dot \gamma_k(\tau)}_{\gamma_k(\tau)}\dd \tau,\\
    				&= M_2(x,2/3) L(\gamma_k)
    			\end{aligned}
    		\end{equation*}
    		If the image of the curve $\gamma_k$ is not fully contained in $\ical_x^{2/3}$, then let 
    		\begin{equation*}
    			\begin{aligned}				
    				t_1 &= \inf\paa{t\in\pac{0,1}: \gamma_k(t)\notin \ical_x^{2/3}},\\
    				t_2 &= \sup\paa{t\in\pac{0,1}: \gamma_k(t)\notin \ical_x^{2/3}}.
    			\end{aligned}
    		\end{equation*}
    		Therefore we can define the tangent space curve $q(t) = R^{-1}_x(\gamma_k(t))$ only for ${t\in\pacp{0,t_1}\cup\papc{t_2,1}}$. Since $y,z\in \ical_x^{1/3}$, it follows that 
    		\begin{align*}
    			\norm{R_x^{-1}(y) - R_x^{-1}(z)}_x \leq 2 \Delta(x) / 3
    		\end{align*}
    		Furthermore, the tangent space curve $q$ traverses the tangent space spherical annulus of width $\Delta(x)/3$ back and forth, hence its length must exceed $2\Delta(x)/3$, i.e.
    		\begin{equation}
    			2/3 \Delta(x) \leq \int_{\pacp{0,t_1}\cup\papc{t_2,1}}\norm{\dot q(\tau)}_x\dd \tau.
    		\end{equation}
    		Hence, we recover the same inequality as before 
    		\begin{align*}
    			\norm{R_x^{-1}(y) - R_x^{-1}(z)}_x &\leq \int_{\pacp{0,t_1}\cup\papc{t_2,1}}\norm{D_2R^{-1}_x(\gamma_k(\tau))\pac{\dot\gamma_k(\tau)}}_x\dd \tau,\\
    			&\leq M_2(x,2/3) \int_{\pacp{0,t_1}\cup\papc{t_2,1}}\norm{\dot \gamma_k(\tau)}_{\gamma_k(\tau)}\dd \tau,\\
    			&\leq M_2(x,2/3) L(\gamma_k).
    		\end{align*}
    		We have shown that for every $k\in\mathbb{N}$ 
    		\begin{equation*}
    			\norm{R_x^{-1}(y) - R_x^{-1}(z)}_x \leq M_2(x,2/3) L(\gamma_k).
    		\end{equation*}
    		Therefore the result remains true upon taking the infimum over $k\in\mathbb{N}$, which yields 		
    		\begin{equation*}
    			\norm{R_x^{-1}(y) - R_x^{-1}(z)}_x \leq M_2(x,2/3) d(x,y)
    		\end{equation*}
    		as desired and shows $M_R(x) = M_2(x,2/3)$.
    	\end{enumerate} \qed
    \end{proof}
    
    The results of Corollary~\ref{cor:lipschitzRetrOnCompact} requires to bound the local Lipschitz constant of the retraction on any compact set $K\subset\mcal$. The following more general result shows this is possible for all local Lipschitz constants of the retractions introduced in Proposition~\ref{prop:lipschRetr}.
    \begin{proposition}\label{prop:lipschitzRetrOnCompactGeneral}
    	For every compact set $K\subset\mcal$ and every $\varepsilon\in\pa{0,1}$, inequalities of Proposition~\ref{prop:lipschRetr} hold for every $x\in K$ with finite strictly positive constants $L_1(K,\varepsilon)$, $L_2(K,\varepsilon)$, $M_1(K,\varepsilon)$, $M_2(K,\varepsilon)$ depending only on $K$.
    \end{proposition}
    The proof of Proposition~\ref{prop:lipschitzRetrOnCompactGeneral} is straightforward once the following result is established.
    \begin{lemma}\label{lem:compactSubsetsOfTcal}
    	For every compact set $K\subset\mcal$ and every $\varepsilon\in\pa{0,1}$, the sets 
    	\begin{equation*}
    		\overline{\ical_K^\varepsilon}:=\underset{x\in K}{\bigcup} \overline{\ical_x^{\varepsilon}}\subset \mcal, \quad
    		\overline{\tcal_K^{\varepsilon}} := \paa{(x,v)\in \tcal: x\in K, \|v\|_x \leq  \varepsilon\Delta(x) } \subset T\mcal,
    	\end{equation*}
    	are compact sets.
    \end{lemma}
    \begin{proof}
    	We have that $E(\overline{\tcal_K^{\varepsilon}}) = K\times \overline{\ical_K^\varepsilon} $, where $E$ is the diffeomorphism of Proposition~\ref{prop:retractionDiffeom}. Therefore, since $K$ is compact, $\overline{\tcal_K^{\varepsilon}}$ is compact if and only if $\overline{\ical_K^\varepsilon}$ is compact. Let us prove $\overline{\tcal_K^{\varepsilon}}$ is sequentially compact, equivalent to being compact by the assumed second countability of the manifold topology. 
    	
    	Consider any sequence $\paa{\pa{x_n,v_n}}_{n\in\Nbb}\subset \overline{\tcal_K^{\varepsilon}}$. Then $\paa{x_n}_{n\in \Nbb}\subset K$ and so there exists a convergent subsequence $\paa{x_{n_k}}_{k\in\Nbb}$ such that $x_{n_k}\convergesTo{k}{\infty} x\in K $.
    	Then there exists $N>0$ such that for every $k>N$, $x_{n_k}\in \ical_{x}^{\varepsilon}$. Therefore the retraction curve $\sigma_k(\tau) = R_x(\tau R^{-1}_x(x_{n_k}))$ is well-defined for all $k > N$. Consider the parallel transport map along $\sigma_k$~\cite[Definition 10.35]{boumalBook} denoted $\parallelTransp{1}{0}{\sigma_k}: T_{x_{n_k}}\mcal\to T_x\mcal$. It is uniquely defined and is an isometry~\cite[Proposition 10.36]{boumalBook}. Let us define $w_k = \parallelTransp{1}{0}{\gamma_{k}}v_{n_k}\in T_x\mcal$.
    	Then since $\norm{v_{n_k}}_{x_{n_k}}\leq \varepsilon \Delta(x_{n_k})$ and $\parallelTransp{1}{0}{\gamma_{k}}$ is an isometry, then $\norm{w_k}_x\leq \varepsilon\Delta(x_{n_k})$. By the continuity of $\Delta$ we know $\paa{\varepsilon\Delta(x_{n_k})}_{k\in \Nbb}$ is a convergent, hence bounded sequence. Therefore $\paa{w_k}_{k\in \Nbb}\subset T_x\mcal$ is a bounded sequence, and thus admits a convergent subsequence $\paa{w_{k_j}}_{j\in \Nbb}$ such that $w_{k_j}\convergesTo{j}{\infty} w\in T_x\mcal$. But since $\norm{w_{k_j}}_x\leq \varepsilon\Delta(x_{n_{k_j}})$, we have that $\norm{w}_{x}\leq \varepsilon\Delta(x) $ and consequently that $(x,w)\in \overline{\tcal_K^{\varepsilon}}$. The standard Riemannian metric on the tangent bundle associated to any Riemannian manifold, also known as Sasaki metric~\cite[p. 80]{sasakiMetric}, allows the definition of a distance function on the tangent bundle defined analogously to~\eqref{eq:distanceFunction}, see \cite[p. 240]{tangentBundleDistance}. Then we have that
    	\begin{equation*}
    		d\pa{(x,w),\pa{x_{n_{k_j}}, v_{n_{k_j}}}} \leq \sqrt{L(\sigma_k)^2 + \norm{w - \parallelTransp{1}{0}{\sigma_{k_j}}v_{n_{k_j}}}_x}
    	\end{equation*}
    	Since the right-hand side converges to zero as $j\longrightarrow +\infty$, this shows that $\paa{(x_{n_{k_j}}, v_{n_{k_j}})}$ converges to $\pa{x,w}\in\overline{\tcal_K^{\varepsilon}}$, concluding the proof. \qed
    \end{proof}
    
    \begin{proof}(of Proposition~\ref{prop:lipschitzRetrOnCompactGeneral})
    	We know that the map $E$ of Proposition~\ref{prop:retractionDiffeom} is a diffeomorphism on	$\tcal$.
    	Therefore, the differentials of the retraction on $\overline{\tcal_K^{\varepsilon}}$ and of the inverse retraction on $\overline{\ical_K^{\varepsilon}}$ are continuous and have continuous operator norm. Since by Lemma~\ref{lem:compactSubsetsOfTcal} these sets are compact, the operator norm attains a (finite) maximum. The constants $L_1(K,\varepsilon)$, $L_2(K,\varepsilon)$, $M_1(K,\varepsilon)$, $M_2(K,\varepsilon)$ are obtained by maximizing $\operatornorm{D_1R}$ and $\operatornorm{D_2R}$ on $\overline{\tcal_K^{\varepsilon}}$ and $\operatornorm{D_1R^{-1}}$ and $\operatornorm{D_2R^{-1}}$ on $\overline{\ical_K^{\varepsilon}}$, respectively. \qed
    \end{proof}
    
    \section{Proof of Lipschitz continuity of RH scheme independent of $h$}\label{app:lipschContIndepOfH}
    The proof of Lemma~\ref{lem:lipschitzContinuityIndepOfH} leverages the Lipschitz continuity of the retraction discussed in Section~\ref{ss:lipschitzContinuityOfRetraction} and detailed Appendix~\ref{app:lipschitzContinuityOfRetraction}. Therefore note that the following sections employ notations introduced at the beginning of Appendix~\ref{app:lipschitzContinuityOfRetraction}.
    
    In this section, we first show the Lipschitz continuity of the $r$-endpoint retraction curve with respect to each of its arguments and then to all its arguments jointly. These preliminary results will be used in the proof of Lemma~\ref{lem:lipschitzContinuityIndepOfH}, reported in Appendix~\ref{app:lipschitzContinuityIndepOfHProof}. 
    \subsection{Preliminary results}\label{app:preliminaryResults}
    \begin{lemma}\label{lem:rEndPointLipschitz}
    	Let $\ucal\subset\mcal$ be retraction-convex such that
    	\begin{equation}\label{eq:lemma24AssumptionOnUcal}
    		\ucal \subset \ical_z^{1/3}, \quad\forall z\in \ucal.
    	\end{equation} The exist positive constants $L_t$, $L_r$ and $L_{xy}$ depending on the retraction verifying:  
    	\begin{enumerate}[label=(\roman*)]
    		\item $d\pa{c_r(t_1;x,y),c_r(t_2;x,y)}\leq L_t d(x,y)|t_1-t_2|$, $\forall\, x,y\in\ucal$, $\forall\, r,t_1,t_2\in\pac{0,1}$.
    		\item $d\pa{c_{r_1}(t;x,y),c_{r_2}(t;x,y)}\leq L_rd(x,y)|r_1-r_2|$, $\forall\, x,y\in\ucal$, $\forall\, r_1,r_2,t\in\pac{0,1}$.
    		\item $d\pa{c_{r}(t;x_1,y_1),c_{r}(t;x_2,y_2)}\leq L_{xy}(d(x_1,x_2)+d(y_1,y_2))$, ${\forall\, x_1,x_2,y_1,y_2\in\ucal}$, $\forall\, r,t\in\pac{0,1}$.
    	\end{enumerate}
    \end{lemma}
    \begin{proof}
    	Given $x,y\in\ucal$ and $r,t\in\pac{0,1}$, first observe that~\eqref{eq:lemma24AssumptionOnUcal} implies that the evaluation of $c_r(t;x,y)$ requires only the evaluation of the retraction on $\tcal_z^{1/3}$ and of the inverse retraction on $\ical_z^{1/3}$ for several $z\in\ucal$. Pick any $z\in\ucal$, denote $K = \overline{\ical_z^{1/3}}$ and consider the retraction's Lipschitz constants $L_1(K,1/3)$, $L_2(K,1/3)$, $M_1(K,1/3)$ and $M_2(K,1/3)$ as introduced in Proposition~\ref{prop:lipschitzRetrOnCompactGeneral} and $L_R(K) = L_2(K,1/3)$ and $M_R(K) = M_2(K,2/3)$ introduced in Corollary~\ref{cor:lipschitzRetrOnCompact}. For conciseness, we simply denote
    	\begin{align*}
    		L_1 = L_1(K,1/3),& \quad L_2 = L_2(K,1/3),\\
    		M_1 = M_1(K,1/3),& \quad M_2 = M_2(K,2/3).
    	\end{align*}
    	and we use the results of Lemma~\ref{lem:lipschRetrAndInvRetr} and Proposition~\ref{prop:lipschRetr} with these constants.
    	\\
    	\emph{(i)} Using Lemma~\ref{lem:lipschRetrAndInvRetr} with the constants defined above
    	\begin{gather}
    		d\pa{c_r(t_1;x,y),c_r(t_2;x,y)} \\
    		\begin{aligned}
    			= d\biggl(&R_{q(r)}\pa{R_{q(r)}^{-1}(x) + t_1\pa{R_{q(r)}^{-1}(y)-R_{q(r)}^{-1}(x)}},\\ &R_{q(r)}\pa{R_{q(r)}^{-1}(x) + t_2\pa{R_{q(r)}^{-1}(y)- R_{q(r)}^{-1}(x)}}\biggr)
    		\end{aligned}\\
    		\leq L_2 \|R_{q(r)}^{-1}(y)-R_{q(r)}^{-1}(x)\|_{q(r)} |t_1 - t_2|\\
    		\leq L_2M_2 d(x,y) |t_1 - t_2|.
    	\end{gather}
    	Therefore we can choose $L_t = L_2 M_2$.\\
    	\emph{(ii)} The smooth curve $\delta(\tau) := c_{r(\tau)}(t;x,y)$, where $r(\tau) = (1-\tau)r_1+\tau r_2$, $\tau\in\pac{0,1}$, joins $c_{r_1}(t;x,y)$ and $c_{r_2}(t;x,y)$. Therefore we can bound
    	\begin{equation}\label{eq:boundOnDistanceCr1_Cr2}
    		d\pa{c_{r_1}(t;x,y),c_{r_2}(t;x,y)} \leq L(\delta) \leq \underset{\tau\in\pac{0,1}}{\max}\paa{\norm{\dot\delta(\tau)}_{\delta(\tau)}}.
    	\end{equation}
    	Denoting $\xi(\tau) := (1-t)R_{q(r(\tau))}^{-1}(x) + t R_{q(r(\tau))}^{-1}(y)\in T_{q(r(\tau))}\mcal$, we have 
    	\begin{equation}
    		\begin{aligned}			
    			\delta(\tau) &= R_{q(r(\tau))}(\xi(\tau)),\\
    			q(r(\tau)) &= R_x(r(\tau)R_x^{-1}(y)).
    		\end{aligned}
    	\end{equation}
    	Therefore
    	\begin{gather}
    		\dot\delta(\tau) = \D_1R_{q(r(\tau))}\pa{\xi(\tau)}\pac{(r_2-r_1)\dot q(r(\tau))} 	+ \D_2R_{q(r(\tau))}\pa{\xi(\tau)}\biggl[\dot\xi(\tau)\biggr],
    	\end{gather}	
    	with 
    	\begin{equation}
    		\dot q(r(\tau)) = \D_2R_{q(r(\tau))}(r(\tau) R_{x}^{-1}(y))\pac{R_{x}^{-1}(y)}
    	\end{equation}
    	and
    	\begin{equation}
    		\dot\xi(\tau)= (1-t)\D_1R^{-1}_{q(r(\tau))}\pa{x}\pac{(r_2 - r_1)\dot q(r(\tau))} +\\ t\D_1R^{-1}_{q(r(\tau))}\pa{y}\pac{(r_2 - r_1)\dot q(r(\tau))}.
    	\end{equation}
    	Using Proposition~\ref{prop:lipschRetr} and Lemma~\ref{lem:lipschRetrAndInvRetr} we can bound the norms of these tangent vectors as follows: 
    	\begin{align}
    		\norm{\dot q(r(\tau))}_{q(r(\tau))} &\leq L_2\|R_x^{-1}(y)\|_x\\
    		&\leq L_2 M_2 d(x,y),
    	\end{align}
    	\begin{align}
    		\norm{\dot\xi(\tau)}_{q(r(\tau))} &\leq \pa{|1-t| + |t|}M_1\norm{\dot q(r(\tau))}_{q(r(\tau))}|r_1-r_2|\\
    		&\leq 2M_1L_2 M_2d(x,y)|r_1-r_2|,
    	\end{align}
    	\begin{align}
    		\norm{\dot\delta(\tau)}_{\delta(\tau)} &\leq L_1\norm{\dot q(r(\tau))}_{q(r(\tau))}|r_1-r_2| + L_2\norm{\dot\xi(\tau)}_{q(r(\tau))}\\
    		&\leq  (L_1 L_2  M_2 + 2L_2^2 M_1 M_2) d(x,y)|r_1-r_2|.
    	\end{align}
    	Together with~\eqref{eq:boundOnDistanceCr1_Cr2}, this shows that $L_r = (L_1 L_2 + 2L_2^2 M_1)M_2$.\\
    	\emph{(iii)} We define a manifold curve joining $c_{r}(t;x_1,y_1)$ and $c_{r}(t;x_2,y_2)$ as 
    	\begin{equation}
    		\delta(\tau) = c_r(t;\delta_x(\tau),\delta_y(\tau)), \tau\in\pac{0,1},
    	\end{equation}
    	where the curves $\delta_x(\tau) := R_{x_1}(\tau R_{x_1}^{-1}(x_2))$, $\delta_y(\tau) := R_{y_1}(\tau R_{y_1}^{-1}(y_2))$ are well-defined since endpoints belong a to retraction-convex set.
    	Then, as previously
    	\begin{equation}
    		d\pa{c_{r}(t;x_1,y_1),c_{r}(t;x_2,y_2)} \leq L(\delta) \leq \underset{\tau\in\pac{0,1}}{\max}\paa{\norm{\dot\delta(\tau)}_{\delta(\tau)}}.
    	\end{equation}
    	We have 
    	\begin{equation}
    		\delta(\tau) = R_{p(\tau)}\pa{\xi(\tau)},
    	\end{equation}
    	where 
    	\begin{align}
    		p(\tau) &:= R_{\delta_x(\tau)}(r R_{\delta_x(\tau)}^{-1}\pa{\delta_y(\tau)}),\\
    		\xi(\tau) &:= (1-t) R_{p(\tau)}^{-1}(\delta_x(\tau)) + t  R_{p(\tau)}^{-1}(\delta_y(\tau)).
    	\end{align}
    	Differentiating these quantities with respect to $\tau$ yields
    	\begin{gather}
    		\begin{aligned}
    			\dot\delta(\tau) &= \D_1R_{p(\tau)}(\xi(\tau))\pac{\dot p(\tau)} + \D_2R_{p(\tau)}(\xi(\tau))\pac{\dot\xi(\tau)},\\ 		
    			\dot p (\tau) &=  \D_1R_{\delta_x(\tau)}(r R_{\delta_x(\tau)}^{-1}\pa{\delta_y(\tau)})\pac{\dot\delta_x(\tau)}+\\
    			&\hspace*{0.6cm}\D_2 R_{\delta_x(\tau)}(r R_{\delta_x(\tau)}^{-1}\pa{\delta_y(\tau)})\biggl[r\D_1R_{\delta_x(\tau)}^{-1}(\delta_y(\tau))\pac{\dot\delta_x(\tau)} +\\ &\hspace*{5.2cm}r\D_2R_{\delta_x(\tau)}^{-1}(\delta_y(\tau))\pac{\dot\delta_y(\tau)}\biggr],\\
    			\dot\xi(\tau) &= (1-t)\biggl(\D_1R_{p(\tau)}^{-1}\pa{\delta_x(\tau)}\pac{\dot p (\tau)} + \D_2R_{p(\tau)}^{-1}\pa{\delta_x(\tau)}\pac{\dot \delta_x (\tau)}\biggr)\\
    			&\hspace*{1.1cm}+ t\biggl(\D_1R_{p(\tau)}^{-1}\pa{\delta_y(\tau)}\pac{\dot p (\tau)} + \D_2R_{p(\tau)}^{-1}\pa{\delta_y(\tau)}\pac{\dot \delta_y (\tau)}\biggr),\\
    			\dot\delta_x(\tau) &= \D_2R_{x_1}\pa{\tau R_{x_1}^{-1}(x_2)}\pac{R_{x_1}^{-1}(x_2)},\\
    			\dot\delta_y(\tau) &= \D_2R_{y_1}\pa{\tau R_{y_1}^{-1}(y_2)}\pac{R_{y_1}^{-1}(y_2)}.
    		\end{aligned}
    	\end{gather}
    	Using Proposition~\ref{prop:lipschRetr} and Lemma~\ref{lem:lipschRetrAndInvRetr} we establish the following bounds: 
    	\begin{align}
    		\norm{\dot\delta_x(\tau)}_{\delta(\tau)}\leq& L_2\norm{R_{x_1}^{-1}(x_2)} \leq  L_2 M_2d(x_1,x_2),\label{eq:lemInterMediateBound1}\\ 		 
    		\norm{\dot\delta_y(\tau)}_{\delta(\tau)}\leq& L_2\norm{R_{y_1}^{-1}(y_2)} \leq  L_2 M_2d(y_1,y_2),\label{eq:lemInterMediateBound2}\\
    		\norm{\dot p(\tau)}_{p(\tau)} \leq& \pa{L_1+L_2M_1}\norm{\dot\delta_x(\tau)}_{\delta_x(\tau)} + L_2M_2\norm{\dot\delta_y(\tau)}_{\delta_y(\tau)},\label{eq:lemInterMediateBound3}\\
    		\norm{\dot\xi(\tau)}_{p(\tau)} \leq& 2M_1\norm{\dot p(\tau)}_{p(\tau)} + M_2\norm{\dot\delta_x(\tau)}_{\delta_x(\tau)} + M_2\norm{\dot\delta_y(\tau)}_{\delta_y(\tau)},\label{eq:lemInterMediateBound4}\\
    		\norm{\dot\delta(\tau)}_{\delta(\tau)} \leq&L_1\norm{\dot p(\tau)}_{p(\tau)} + L_2\norm{\dot\xi(\tau)}_{p(\tau)}.\label{eq:lemFinalBound}
    	\end{align}
    	By suitably plugging the previous inequalities into the right-hand side of~\eqref{eq:lemFinalBound}, we find 
    	\begin{equation*}
    		\norm{\dot\delta(\tau)}_{\delta(\tau)} \leq L_x d(x_1,x_2) + L_y d(y_1, y_2)
    	\end{equation*}
    	where $L_x$ and $L_y$ are polynomials of $L_1, L_2, M_1$ and $M_2$.
    	Then, we can take $L_{xy} = \max\paa{L_x,L_y}$. \qed
    \end{proof}
    \begin{corollary}\label{cor:rEndPointLipschitz}
    	Let $\ucal\subset \mcal$ be any retraction-convex set as in Lemma~\ref{lem:rEndPointLipschitz}. For any ${x_1,x_2,y_1,y_2\in\ucal}$  and for every $r_1,r_2,t_1,t_2\in\pac{0,1}$, we have
    	\begin{equation}
    		\begin{aligned}			
    			d\pa{c_{r_1}(t_1;x_1,y_1),c_{r_2}(t_2;x_2,y_2)}\leq &\pa{d(x_1,y_1) + d(x_2,y_2)}\pa{\oneOver{2}L_t|t_1-t_2| + \oneOver{2}L_r|r_1-r_2|} \\ + &\pa{d(x_1,x_2) + d(y_1,y_2)}L_{xy}.
    		\end{aligned}
    	\end{equation}
    \end{corollary}
    \begin{proof}
    	Using the triangular inequality and Lemma~\ref{lem:rEndPointLipschitz} we have 
    	\begin{align}
    		&\begin{aligned}
    			&d\pa{c_{r_1}(t_1;x_1,y_1),c_{r_2}(t_2;x_2,y_2)}\\  \leq& d\pa{c_{r_1}(t_1;x_1,y_1),c_{r_1}(t_2;x_1,y_1)} + d\pa{c_{r_1}(t_2;x_1,y_1),c_{r_2}(t_2;x_2,y_2)}
    		\end{aligned}\\
    		&\begin{aligned}
    			\leq &L_td(x_1,y_1)|t_1-t_2| + d\pa{c_{r_1}(t_2;x_1,y_1),c_{r_1}(t_2;x_2,y_2)} \\&+ d\pa{c_{r_1}(t_2;x_2,y_2),c_{r_2}(t_2;x_2,y_2)}
    		\end{aligned}\\
    		&\hspace*{-0.1cm}\leq\hspace*{-0.1cm} L_td(x_1,y_1)|t_1-t_2| + L_{xy} \pa{d(x_1,x_2) + d(y_1, y_2)} + L_r d(x_2,y_2) |r_1-r_2|.\label{eq:corBound1}
    	\end{align}
    	Exchanging $(r_1,t_1,x_1, y_1)$ and $(r_2, t_2, x_2, y_2)$ and repeating the same procedure leads to 
    	\begin{equation}\label{eq:corBound2}		
    		\begin{aligned}
    			d\pa{c_{r_2}(t_2;x_2,y_2),c_{r_1}(t_1;x_1,y_1)} \leq &L_td(x_2,y_2)|t_1-t_2| + L_{xy} \pa{d(x_1,x_2) + d(y_1, y_2)}\\ + &L_r d(x_1,y_1) |r_1-r_2|.
    		\end{aligned}
    	\end{equation}	
    	Then averaging~\eqref{eq:corBound1} and~\eqref{eq:corBound2} proves the result. \qed
    \end{proof}
    
    \subsection{Proof of Lemma~\ref{lem:lipschitzContinuityIndepOfH}}\label{app:lipschitzContinuityIndepOfHProof}
    The following proof considers the RH interpolant given in Definition~\ref{def:genCasteljau}, that is with the particular choice~\eqref{eq:bestChoiceOfr1r01r12r012} for the functions $r_1,r_{01},r_{12}$, and $r_{012}$. Nevertheless, the result remains valid for any choice of these functions, as long as they are Lipschitz continuous.
    \begin{proof} (of Lemma~\ref{lem:lipschitzContinuityIndepOfH})
    	In the proof of Proposition~\ref{prop:wellDefinitessOfRHOfCurve}, we show the RH interpolant $H_h$ defined in~\eqref{eq:hh} is well-defined provided $h < h_1$. In fact, then all control points to construct $H_h$ are in contained in the retraction-convex set $B(\gamma(t),Q h)\subset B(\gamma(t), \rho_{\min})$, for some $Q>0$ independent of $h$. For every $x\in\mcal$, define 
    	\begin{equation}
    		\bar\nu(x) = \sup\paa{\nu > 0: B(x,\nu)\subset \ical_x^{1/3}}.
    	\end{equation} 
    	The constant $\nu_{\min} := \inf_{\tau\in\pac{0,1}}{\bar\nu(\gamma(\tau))}$ is strictly positive since, if it was zero, this would violate the Lipschitz continuity of the retraction on the compact set $\gamma\pa{\pac{0,1}}$, guaranteed by Proposition~\ref{prop:lipschitzRetrOnCompactGeneral}.
    	Therefore, if we require $h < h_3:=\min\paa{h_1, \frac{\nu_{\min}}{2 Q}}$, then for any $x,y\in B(\gamma(t),Qh)$ we have 
    	\begin{equation*}
    		d(x,y) \leq d(x,\gamma(t)) + d(y,\gamma(t)) < \nu_{\min}.
    	\end{equation*}
    	Therefore, if $h<h_3$, all control points defining $H_h$ are contained into a retraction-convex set verifying~\eqref{eq:lemma24AssumptionOnUcal}.
    	Hence, in the following we can use the constants $L_t$, $L_r$ and $L_{xy}$ given by Corollary~\ref{cor:rEndPointLipschitz}. 
    	Let us denote 
    	\begin{equation*}
    		p_0 = \gamma(t),\quad v_0  = \dot\gamma(t),\quad p_1 = \gamma(t+h),\quad v_1 = \dot\gamma(t+h).
    	\end{equation*}
    	From Corollary~\ref{cor:hermiteRetrFullSolution}, we have that
    	\begin{equation}
    		H_h(\tau) = \alpha\pa{\frac{\tau-t}{h};p_0, hv_0, p_1, hv_1}.
    	\end{equation}
    	We shall prove that there exists $L_{RH}$ (depending explicitly on $L_t$, $L_r$ and $L_t$) such that for any $z_1,z_2\in\pac{0,1}$
    	\begin{equation}\label{eq:alphaLipCont}
    		d(\alpha\pa{z_1;p_0, hv_0, p_1, hv_1}, \alpha\pa{z_2;p_0, hv_0, p_1, hv_1}) \leq L_{RH} h |z_1-z_2|.
    	\end{equation}
    	Then using~\eqref{eq:alphaLipCont}, we can conclude that for any $\tau_1,\tau_2\in\pac{t,t+h}$ 
    	\begin{align}
    		d(H_h(\tau_1),H_h(\tau_2)) &= d\pa{\alpha\pa{\frac{\tau_1-t}{h};p_0, hv_0, p_{1}, hv_{1}},\alpha\pa{\frac{\tau_2-t}{h};p_0, hv_0, p_{1}, hv_{1}}}\\
    		&\leq L_{RH} h \frac{|\tau_1 - \tau_2|}{h} = L_{RH} |\tau_1-\tau_2|.
    	\end{align}
    	Let us now prove~\eqref{eq:alphaLipCont} by unfolding the recursive definition of $\alpha$.
    	\begin{equation}
    		\alpha\pa{z;p_0, hv_0, p_{1}, hv_{1}} = \beta_{012}(z,z,z) = c_{z}(z; \beta_{01}(z,z), \beta_{12}(z,z)).
    	\end{equation} 
    	Applying several times Corollary~\ref{cor:rEndPointLipschitz} we find
    	\begin{equation}\label{eq:mainInequalityToProve}
    		\begin{aligned}
    			&d(c_{z_1}(z_1; \beta_{01}(z_1,z_1), \beta_{12}(z_1,z_1)),c_{z_2}(z_2; \beta_{01}(z_2,z_2), \beta_{12}(z_2,z_2)))\\
    			&\leq |z_1 - z_2| \frac{L_t + L_r}{2} \pa{d(\beta_{01}(z_1, z_1), \beta_{12}(z_1,z_1)) + d(\beta_{01}(z_2, z_2), \beta_{12}(z_2,z_2))} \\
    			&\quad+ L_{xy} \pa{d(\beta_{01}(z_1,z_1),\beta_{01}(z_2, z_2)) + d(\beta_{12}(z_1,z_1),\beta_{12}(z_2, z_2))},
    		\end{aligned}
    	\end{equation}
    	\begin{equation}\label{eq:intermediateInequality1_1}
    		\begin{aligned}
    			d(\beta_{01}(z_1, z_1), \beta_{12}(z_1,z_1)) &= d(c_{0}(z_1, \beta_0(z_1), \beta_1(z_1)), c_{1}(z_1, \beta_1(z_1), \beta_2(z_1)))\\
    			&\leq \pa{\frac{1}{2} L_r + L_{xy}}\pa{d(\beta_0(z_1), \beta_1(z_1)) + d(\beta_1(z_1), \beta_2(z_1))},
    		\end{aligned}
    	\end{equation}
    	\begin{equation}\label{eq:intermediateInequality1_2}
    		\begin{aligned}
    			d(\beta_{01}(z_2, z_2), \beta_{12}(z_2,z_2)) \leq \pa{\frac{1}{2} L_r + L_{xy}}\pa{d(\beta_0(z_2), \beta_1(z_2)) + d(\beta_1(z_2), \beta_2(z_2))},
    		\end{aligned}
    	\end{equation}
    	\begin{equation}\label{eq:intermediateInequality1_3}
    		\begin{aligned}
    			d(\beta_{01}(z_1, z_1), \beta_{01}(z_2,z_2)) &= d(c_{0}(z_1, \beta_0(z_1), \beta_1(z_1)), c_{0}(z_2, \beta_0(z_2), \beta_1(z_2)))\\
    			&\leq\frac{L_t}{2}|z_1 - z_2|\pa{d(\beta_0(z_1), \beta_1(z_1) + d(\beta_0(z_2),\beta_1(z_2))} \\
    			&\quad+ L_{xy}\pa{d(\beta_0(z_1), \beta_0(z_2)) + d(\beta_1(z_1), \beta_1(z_2))},
    		\end{aligned}
    	\end{equation}
    	\begin{equation}\label{eq:intermediateInequality1_4}
    		\begin{aligned}
    			d(\beta_{12}(z_1, z_1), \beta_{12}(z_2,z_2)) 
    			&\leq\frac{L_t}{2}|z_1 - z_2|\pa{d(\beta_1(z_1), \beta_2(z_1) + d(\beta_1(z_2),\beta_2(z_2))} \\
    			&\quad+ L_{xy}\pa{d(\beta_1(z_1), \beta_1(z_2)) + d(\beta_2(z_1), \beta_2(z_2))}.
    		\end{aligned}
    	\end{equation}
    	Plugging~\eqref{eq:intermediateInequality1_1}-\eqref{eq:intermediateInequality1_4} in~\eqref{eq:mainInequalityToProve} and rearranging terms yields
    	\begin{equation}\label{eq:intermediateInequality2}
    		\begin{aligned}
    			d(c_{z_1}(z_1; \beta_{01}(z_1,z_1), \beta_{12}(z_1,z_1)),c_{z_2}(z_2; \beta_{01}(z_2,z_2), \beta_{12}(z_2,z_2)))\\ \leq|z_1 - z_2|\frac{(L_t + L_r)(L_r + 2 L_{xy}) + 2L_{xy}L_t}{4} \\
    			\cdot\pa{d(\beta_0(z_1), \beta_1(z_1)) + d(\beta_1(z_1), \beta_2(z_1)) + d(\beta_0(z_2), \beta_1(z_2)) + d(\beta_1(z_2), \beta_2(z_2))}\\
    			+ L_{xy}^2\pa{d(\beta_0(z_1), \beta_0(z_2)) + 2d(\beta_1(z_1), \beta_1(z_2))+ d(\beta_2(z_1), \beta_2(z_2))}.
    		\end{aligned}
    	\end{equation}
    	We now bound the seven distance function evaluations using once again Corollary~\ref{cor:rEndPointLipschitz}.
    	\begin{align}
    		d(\beta_0(z_1), \beta_1(z_1)) &\leq \pa{\frac{L_r}{4} + L_{xy}}\pa{d(p_0, p_0^+) + d(p_0^+, p_1^-)},  \\			
    		d(\beta_0(z_2), \beta_1(z_2)) &\leq \pa{\frac{L_r}{4} + L_{xy}}\pa{d(p_0, p_0^+) + d(p_0^+, p_1^-)},  \\			
    		d(\beta_1(z_1), \beta_2(z_1)) &\leq \pa{\frac{L_r}{4} + L_{xy}}\pa{d(p_0^+, p_1^-) + d(p_1^-, p_1)},\\			
    		d(\beta_1(z_2), \beta_2(z_2)) &\leq \pa{\frac{L_r}{4} + L_{xy}}\pa{d(p_0^+, p_1^-) + d(p_1^-, p_1)},\\
    		d(\beta_0(z_1), \beta_0(z_2)) &\leq L_t d(p_0, p_0^+)|z_1 - z_2|,  \\
    		d(\beta_1(z_1), \beta_1(z_2)) &\leq L_t d(p_0^+, p_1^-)|z_1 - z_2|,  \\
    		d(\beta_2(z_1), \beta_2(z_2)) &\leq L_t d(p_1^-, p_1)|z_1 - z_2|.  
    	\end{align}
    	Inserting these bounds in~\eqref{eq:intermediateInequality2} provides a constant $\tilde L>0$ depending only on $L_t$, $L_r$ and $L_{xy}$ such that 
    	\begin{equation}\label{eq:lipContAlphaAlmostFinal}
    		\begin{aligned}
    			d(\alpha\pa{z_1;p_0, hv_0, p_1, hv_{1}}, \alpha\pa{z_2;p_0, hv_0, p_1, hv_{1}}) \\
    			\leq (d(p_0, p_0^+) + d(p_0^+, p_1^-) + d(p_1^-, p_1))\tilde L |z_1 - z_2|.
    		\end{aligned}
    	\end{equation}
    	Using Corollary~\ref{cor:rEndPointLipschitz}-\ref{cor:lipschRetr} and denoting $L_\gamma$ the Lipschitz constant of the curve $\gamma$ we find
    	\begin{align}
    		d(p_0, p_0^+) &= d\pa{\gamma\pa{t}, R_{\gamma(t)}\pa{\frac{h}{3}\dot\gamma(t)}}\\ 
    		&\leq L_2 \norm{\frac{h}{3}\dot\gamma(t)}\leq \frac{h}{3}L_2 L_\gamma,  \\
    		d(p_1^-, p_1) &= d\pa{R_{\gamma(t+h)}\pa{-\frac{h}{3}\dot\gamma(t+h)}, \gamma(t+h)} \\ &\leq L_2 \norm{\frac{h}{3}\dot\gamma(t+h)}\leq \frac{h}{3}L_2 L_\gamma,	\\
    		d(p_0^+, p_1^-) &\leq d(p_0^+, p_0) + d(p_0, p_1) + d(p_1, p_1^-)\\
    		&\leq  \frac{2h}{3}L_2L_\gamma + hL_\gamma .
    	\end{align}
    	Finally, plugging these bounds into~\eqref{eq:lipContAlphaAlmostFinal} proves~\eqref{eq:alphaLipCont} with $L_{RH} = \tilde L L_\gamma \pa{\frac{4}{3}L_2 + 1}$ and concludes the proof. \qed
    \end{proof}

\end{document}